\documentclass[10pt]{amsart}
\usepackage{amsmath,amssymb,amsfonts,amsthm,amsopn}
\usepackage{latexsym}
\usepackage{mathrsfs}
\usepackage{xcolor}
\usepackage{color, colortbl}
\usepackage{pb-diagram}
\usepackage[title]{appendix}
\usepackage{tikz}
\usepackage{mathtools}
\usepackage{graphicx}
\usepackage{subcaption}
\usepackage{enumerate}

\mathtoolsset{showonlyrefs}

\usepackage{multirow}

\newcommand{\tfs}{time-frequency shift}

\newtheorem{theorem}{Theorem}[section]
\newtheorem{lemma}[theorem]{Lemma}
\newtheorem{corollary}[theorem]{Corollary}
\newtheorem{proposition}[theorem]{Proposition}
\newtheorem{definition}[theorem]{Definition}

\newtheorem{cor}[theorem]{Corollary}
\newtheorem{remark}[theorem]{Remark}
\theoremstyle{definition}
\newtheorem{example}[theorem]{Example}

\newcommand{\beqa}{\begin{eqnarray*}}
	\newcommand{\eeqa}{\end{eqnarray*}}

\DeclareMathOperator*{\Sp}{Sp}
\DeclareMathOperator*{\Mp}{Mp}
\DeclareMathOperator*{\Sym}{Sym}

\DeclareMathOperator*{\GL}{GL}

\newcommand{\field}[1]{\mathbb{#1}}
\newcommand{\bR}{\field{R}}        
\newcommand{\bN}{\field{N}}        
\newcommand{\bC}{\field{C}}        
\newcommand{\bF}{\field{F}}        
\def\la{\lambda}

\def\eps{\varepsilon}

\def\cF{\mathscr{F}}              
\def\cS{\mathcal{S}}
\def\cD{\mathcal{D}}

\def\cA{\mathcal{A}}

\def\cC{\mathcal{C}}

\def\cR{\mathcal{R}}

\def\a{\aleph}
\def\hf{\hat{f}}
\def\hg{\hat{g}}

\def\rd{\bR^d}

\def\rdd{{\bR^{2d}}}

\def\lrd{L^2(\rd)}

\def\intrd{\int_{\rd}}
\def\intrdd{\int_{\rdd}}

\def\R{\right)}

\def\<{\left<}
\def\>{\right>}

\def\mv1{M_v^1}

\def\phas{(x,\xi )}
\def\mn{(m,n)}
\def\mn'{(m',n')}

\newcommand{\norm}[1]{\lVert#1\rVert}

\def\o{\xi}
\def\a{\alpha}

\def\R{\mathbb{R}}
\def\Ren{\mathbb{R}^d}

\def\f{\varphi}

\def\Sn2{S_{2}(L^{2}(\Ren))}
\def\S1{S_{1}(L^{2}(\Ren))}
\def\sig00{\sigma_{0,0}}

\def\la{\langle}
\def\ra{\rangle}

\begin{document}
	
\begin{abstract}
	We study the Wigner kernel and the Gabor matrix associated with the propagators of a broad class 
	of linear evolution equations, including the complex heat, wave, 
	and  Hermite equations. Within the framework of time-frequency analysis, we derive 
	explicit expressions for the Wigner kernels of Fourier multipliers and establish quantitative 
	decay estimates for the corresponding Gabor matrices. These results are obtained under symbol 
	regularity conditions formulated in the Gelfand-Shilov scale and ensure exponential off-diagonal 
	decay or quasi-diagonality of the matrix representation. We believe this approach can be extended to more general  symbols in the pseudodifferential setting, improving the existing results in terms of their Gabor matrix decay.
	
	For the \emph{complex heat equation}, we obtain closed-form formulas exhibiting both dissipative 
	and oscillatory behavior governed respectively by the real and imaginary parts of the diffusion 
	parameter. The modulus of the Gabor matrix is shown to display Gaussian decay and temporal 
	spreading consistent with diffusion phenomena. In contrast, the \emph{complex Hermite equation} 
	is analyzed via H\"ormander's metaplectic semigroup, where the propagator decomposes as the product 
	of a real Hermite semigroup and a fractional Fourier transform. In this setting, the Gabor matrix 
	retains its Gaussian shape while undergoing a pure rotation on the time-frequency plane, 
	reflecting the symplectic structure of the underlying flow.
	
	The analysis provides a unified operator-theoretic and phase-space perspective on parabolic 
	and hyperbolic evolution equations, linking the geometry of their symbols with the 
	sparsity and localization properties of their Gabor representations. Explicit formulas are 
	given in a form suitable for numerical computation and visualization of phase-space dynamics.
\end{abstract}

	\title{Wigner and Gabor Phase-Space Analysis of Propagators for Evolution Equations}
	\author{Elena Cordero}
\address{Universit\`a di Torino, Dipartimento di Matematica, via Carlo Alberto 10, 10123 Torino, Italy}
\email{elena.cordero@unito.it}
\author{Gianluca Giacchi}
\address{Università della Svizzera italiana, East Campus, Sector D, Via la Santa 1, 6962 Lugano, Switzerland.}
\email{gianluca.giacchi@usi.ch}
\author{Luigi Rodino}
\address{Universit\`a di Torino, Dipartimento di Matematica, via Carlo Alberto 10, 10123 Torino, Italy}
\email{luigi.rodino@unito.it}
	
	\thanks{}


\maketitle

\section{Introduction}
The Wigner distribution, introduced by Wigner in 1932 in the context of quantum physics as a joint quasiprobability distribution for position and momentum with applications to Schr\"odinger-type problems~\cite{Wigner42}, has been refined and extended over subsequent decades, leading to generalizations such as the cross-Wigner distribution introduced in~\cite{Szu}:
$$
W(f, g)(x, \xi) = \int_{\mathbb{R}^d} f\left(x + \frac{y}{2}\right)\overline{g\left(x - \frac{y}{2}\right)}e^{-2\pi i \xi \cdot y}\, dy, \qquad (x, \xi) \in \mathbb{R}^{2d}, \; f, g \in L^2(\mathbb{R}^d).
$$
We write $Wf = W(f, f)$. This era marked a pinnacle for time-frequency representations. Following Gabor's seminal work of 1946~\cite{Gabor}, the short-time Fourier transform (STFT)
$$
V_g f(x, \xi) = \int_{\mathbb{R}^d} f(y)\overline{g(x - y)}e^{-2\pi i \xi \cdot y}\, dy, \qquad (x, \xi) \in \mathbb{R}^{2d}, \; f, g \in L^2(\mathbb{R}^d)
$$
emerged as a powerful tool for decomposing and analyzing signals in the phase-space domain using time-frequency shifts
$$
\pi(x, \xi)g(y) = e^{2\pi i \xi y}g(y - x), \qquad (x, \xi) \in \mathbb{R}^{2d}, \; g \in L^2(\mathbb{R}^d).
$$
Within this framework, \emph{Gabor matrices} provide a natural representation of operators with respect to time-frequency shifts. For a linear continuous operator $T : \mathcal{S}(\mathbb{R}^d) \to \mathcal{S}'(\mathbb{R}^d)$, $g \in \mathcal{S}(\mathbb{R}^d) \setminus \{0\}$,  and dual window $\gamma$ with $\la g,\gamma\ra=1$, 
the time-frequency representation of $T$ is
\begin{equation}\label{GMat1}
	Tf(x)=\intrdd\intrdd\la T\pi(z)g,\pi(w)g\ra\la f,\pi(z)\gamma\ra \pi(w)\gamma\, dzdw, \qquad f\in \cS(\rd),
\end{equation}
and the coefficients
\begin{equation}\label{defGMat}
	h(z,w)=\la T\pi(z)g,\pi(w)g\ra=V_g(T\pi(z)g)(w), \qquad z,w\in\rdd,
\end{equation}
define the so-called {Gabor matrix} of $T$, acting analogously to matrices in the Euclidean framework.

Additional structures on these matrices greatly facilitate the numerical analysis of linear operators. Sparsity plays a pivotal role, as sparse matrices exhibit lower computational complexity, enabling tailored efficient algorithms. The counterpart of sparsity for Gabor matrices is \emph{off-diagonal decay} estimates, such as
\begin{equation}\label{off-diag}
	|\la T\pi(z)g,\pi(w)g\ra|\leq Ce^{-\varepsilon|z-w|^{1/\mu}}, \qquad C>0,\, \mu>0,
\end{equation}
 see \cite{Elena-book} and reference there, in particular \cite{CNRTAMS2015}. In view of \eqref{defGMat}, Gabor matrices are related to the STFT, while a more Wigner distribution-oriented perspective was proposed more recently in \cite{CRGFIO1}. Following the original idea of Wigner, for a given linear operator $T$ continuous from $\cS(\rd)$ to $\cS'(\rd)$ there exists an operator $K:\cS(\rdd)\to\cS'(\rdd)$ that we may refer to as {\em Wigner operator}, such that the intertwining relation
\begin{equation}
	W(Tf,Tg)=KW(f,g), \qquad f,g\in\cS(\rd)
\end{equation}
holds. The kernel $k_W\in\cS'(\bR^{4d})$ of $K$ is called the {\em Wigner kernel} of $T$ and it is therefore characterized by
\begin{equation}
	W(Tf,Tg)(z)=\int_{\rdd}k_W(z,w)W(f,g)(w)dw,\qquad f,g\in\cS(\rd),
\end{equation}
where the integral must be interpreted in the distributional sense. {\em Wigner analysis} consists of adopting a time-frequency perspective on $T$, possibly a hyperbolic or a parabolic operator, and in obtaining information about $Tf$ directly from $W(Tf)$. 

The Wigner kernel of $T$ and its Gabor matrix are related by
\begin{equation}\label{convgg1}
	|\la T\pi(z)g,\pi(w)g\ra|^2=k_W\ast (Wg\otimes Wg)(z,w), \qquad z,w\in\rdd,
\end{equation}
we address \cite{CGR2025} for the details. In particular, by choosing $g$ as a Gaussian, we see that off-diagonal estimates such as \eqref{off-diag} correspond to {\em quasi-diagonality} in the Wigner kernel perspective. Recall that a kernel $k\in\cS'$ is quasi-diagonal if $k\ast\Phi$, where $\Phi$ is a Gaussian, is concentrated on the diagonal $\Delta=\{(x,y):x=y\}$ \cite{GR2025}. Quasi-diagonality is therefore another manifestation of sparsity conditions for operators within a functional framework. Two aspects prevent pointwise estimates of the Wigner kernel: the fact that the Wigner kernel is generally a tempered distribution and the presence of the ghost frequencies, inevitably introduced by the Wigner distribution $W(Tf)$. Therefore, the filtering in \eqref{convgg1} is not only essential to obtain pointwise estimates for the Wigner kernel, but also fundamental in mitigating ghost frequencies. 

When $T$ is a Fourier integral operator (FIO) or the propagator of a Schrödinger equation with quadratic Hamiltonian and bounded perturbation, off-diagonal decay is generally violated. Nevertheless, a similar form of localization can still be established, in which the Gabor matrix is concentrated along the graph of the Hamiltonian flow. In general, it is possible to obtain decay estimates in the form
\begin{equation}
	|\la T\pi(z)g,\pi(w)g\ra|\leq C 
	e^{-\varepsilon|w-\chi z|^{1/\mu}}, 
	\quad z,w\in\rdd,
\end{equation}
for $C,\varepsilon,\mu>0$, where $\chi$ denotes a $2d\times2d$ (real) symplectic matrix describing the underlying Hamiltonian flow.
See \cite{Elena-book} with particular reference to \cite{CNRex2015}. Corresponding estimates for the Wigner kernel of Schrödinger propagators are in \cite{CGP2025,CRGFIO1,CGRPartII}.

In this work, addressing the off-diagonal decay \eqref{off-diag} and the quasi-diagonality property, we consider the Wigner kernel and the Gabor matrix of the propagators for the complex heat equation
\begin{equation}\label{LRintroA1}
	\begin{cases}
		\partial_tu=\gamma\Delta_xu & \text{$t\geq0$, $x\in\rd$},\\
		u(0,x)=u_0(x),
	\end{cases}
\end{equation}
with $\gamma=\alpha+i\beta\in\mathbb{C}\setminus\{0\}$, $\alpha>0$, and for the wave equation
\begin{equation}\label{LRintroA2}
	\begin{cases}
		\partial_{tt}u-\Delta_xu=0 & \text{$t\in\bR$, $x\in\rd$},\\
		u(0,x)=u_0(x),\\
		\partial_tu(0,x)=u_1(x).
	\end{cases}
\end{equation}
For both these problems, the propagator can be expressed in terms of a Fourier multiplier with symbol $\sigma(t,\xi)$, $\xi\in\rd$. Specifically, for \eqref{LRintroA1},
\begin{equation}\label{LRintroA3}
	\sigma(t,\xi)=e^{-4\pi^2\gamma t|\xi|^2}, \qquad t>0, \, \xi\in\rd,
\end{equation}
and for \eqref{LRintroA2}, under the assumption $u_0=0$,
\begin{equation}\label{LRintroA4}
	\sigma(t,\xi)=\frac{\sin(2\pi|\xi|t)}{2\pi|\xi|}, \qquad t\in\bR, \, \xi\in\rd\setminus\{0\}.
\end{equation}
Both such symbols can be seen as elements of the H\"ormander's class $S^0_{0,0}(\rdd)$, hence a preliminary information about the Wigner kernel is provided by \cite{CRPartI}, and about the Gabor matrix by \cite{Elena-book}. The analysis here will be deeper, addressing the peculiarities of \eqref{LRintroA3}, \eqref{LRintroA4} in the framework of classes of Fourier multipliers. 
Namely, in Section~\ref{sec:Multipliers}, we derive general formulas for the Wigner kernel and Gabor matrix of operators $\sigma(t, D_x)$ with $\sigma(t, \cdot) \in \mathcal{S}'(\mathbb{R}^d)$, see Theorem \ref{teo:3.3} and \eqref{eq:wigkernel} below. Motivated by the core examples  \eqref{LRintroA3} and \eqref{LRintroA4}, we shall introduce in Section~\ref{sec:DecayingSymbols} the classes of symbols $\sigma(\xi)$ satisfying, for $\mu\geq0$, $\nu>0$,
\begin{equation}\label{LRintroA5}
	|\partial^\alpha_\xi\sigma(\xi)|\leq C^{|\alpha|+1}(\alpha!)^\mu e^{-\varepsilon|\xi|^{1/\nu}}, \qquad \xi\in\rd,
\end{equation}
valid for \eqref{LRintroA3} with $\mu=\nu=1/2$, and the larger class
\begin{equation}\label{LRintroA6}
	|\partial^\alpha_\xi\sigma(\xi)|\leq C^{|\alpha|+1}(\alpha!)^\mu, \qquad \xi\in\rd,
\end{equation}
valid for \eqref{LRintroA4} with $\mu=0$, with $|\xi|$ large. For both these classes, we have off-diagonal behavior \eqref{off-diag} for $\sigma(D_x)$, with the basic remark that, for $\sigma(\xi)$ satisfying the more restrictive condition \eqref{LRintroA5}, we obtain an additional regularization with respect to the dual variable, namely for the Gabor matrix (see Theorem \ref{teor:Luigi}, estimate \eqref{Luigi2}):
\begin{equation}\label{LRintroA7}
	|\la \sigma(D_x)\pi(z)g,\pi(w)g \ra|\leq Ce^{-\varepsilon(|\xi|^2+|\eta|^2)^{\rho/2}}e^{-\varepsilon|x-y|^{1/\mu}},
\end{equation}
$z=(x,\xi)\in\rdd$, $w=(y,\eta)\in\rdd$. 
As for the Wigner kernel, it is of the form
$$k_W(z,w)=\delta(\xi-\eta) a(x,\xi,y,\eta),$$
where $a(x,\xi,y,\eta)$ under the assumption \eqref{LRintroA5} satisfies the same estimates as in \eqref{LRintroA7}, see Remark \ref{rem:4.2}. A striking result is that in the case $\mu=0$ in \eqref{LRintroA6} the kernel $k_W(z,w)$ may present singularities, but it is compactly supported near the diagonal $|z-w|\leq \text{const}$,
see Theorem \ref{LRthm2circ}.

Explicit expressions, adapted to numerical computations, for \eqref{LRintroA3}, \eqref{LRintroA4} are given in Sections $5$ and $6$, representing the core of the paper. For the complex heat equation in Section $5$, the Gabor matrix is computed first, with evaluations in the constants in the Gaussian decay \eqref{LRintroA7} and display of the temporal spreading, see Proposition \ref{prop:5.3}. The Wigner kernel, in the purely parabolic case $\beta=0$, is given by 
$$
k_W(x,\xi,y,\eta)=\delta(\xi-\eta)\,
(8\pi \alpha t)^{d/2}
\exp\left(-\frac{|x-y|^2}{2\alpha t}-8\pi^2\,\alpha t |\xi|^2\right),
$$
whereas the expression for $\beta\not=0$ exhibits both the dissipative and oscillatory behavior, see \eqref{WignerOpHeat}. 

The analysis of the wave propagator in Section $6$ starts with the computation of the Wigner kernel. Special attention is devoted to the case $d\leq 3$, see Theorem \ref{thr:wave-correct}.  For $d=1$ one obtains 
$$ k_W(x,\xi,y,\eta)=\delta(\xi-\eta)\,\chi_{|x-y|\leq t}\frac{\sin(4\pi(t-|x-y|)\xi)}{4\pi\xi},
$$
with $\chi$ characteristic function. Similar expressions are obtained for $d=2,d=3$ in terms of the Bessel function $J_0$. The Gabor matrix is then computed, with particular attention to the classical problem of the lacunas, cf. \cite{ABG1970,ABG1973,CLT2024}. Namely, for  $d=1$, $d=3$ the kernel of the propagator of the wave equation vanishes in some interior regions. The Wigner transform produces ghost frequencies in such lacunas, whereas the Gabor matrix re-establishes the behavior of the kernel of the propagator, by exhibiting there a rapid decay. \par

Sections~7 and~8 are devoted to the alternative approach provided by the use of complex metaplectic operators. Standard metaplectic operators in $Mp(d, \mathbb{R})$ were used in \cite{CGRPartII} to compute 
the Wigner kernel of Schrödinger propagators. 
When dealing with the parabolic case, it is natural to refer to operators in $Mp_+(d, \mathbb{C})$ 
of Hörmander~\cite{Hormander1995}, quantization of the symplectic matrices in $Sp_+(d, \mathbb{C})$, 
sub-semigroup of the complex symplectic group $Sp(d, \mathbb{C})$.

Recalling the basic facts of the corresponding calculus, we recapture in this setting 
the results of Section~$5$, and obtain new results in Section~$8$ for the complex Hermite equation:
\[
\begin{cases}
	\partial_t u = (\theta + i\mu) \left( \frac{1}{4\pi} \Delta - \pi |x|^2 \right) u, \\[6pt]
	u(0, x) = u_0(x),
\end{cases}
\]
with $\theta > 0$ and $\mu \in \mathbb{R}$. 

In terms of metaplectic operators, the Gabor matrix is expressed by Corollary \ref{cor:8.1}, 
and the Wigner kernel by Corollary~\ref{cor:8.2}. 
The Gabor matrix satisfies global Gaussian estimates, 
with a rotation on the time-frequency plane, 
reflecting the presence of the parameter $\mu \neq 0$.\par

 We finally observe that the results of Sections \ref{sec:Multipliers} and \ref{sec:DecayingSymbols} allow applications to more general classes of Cauchy problems for constant-coefficient equations of hyperbolic or parabolic type. Namely,  the operators are given by
\begin{equation}\label{LRintroA8}
P(\partial_t,D_x) = \partial_t^m + \sum_{k=1}^m a_k(D_x)\,\partial_t^{m-k}, 
\qquad (t,x)\in \R\times\R^d,
\end{equation}
where the coefficients $a_k(\xi)$ are polynomials of arbitrary degree, not necessarily homogeneous. Considering the forward Cauchy problem with Schwartz initial data, well-posedness is ensured by the forward Hadamard-Petrowsky condition: 
\begin{equation}\label{LRintroA9}
	(\tau,\zeta)\in\bC\times\rd, \quad P(i\tau,\zeta)=0 \quad \Rightarrow \quad \Im(\tau)\geq-C
\end{equation}
for some $C>0$.

Fourier analysis shows that the solution can be written in terms of the fundamental solution 
\begin{equation}\label{eq:convkernel}
	E(t,x) = \cF^{-1}_{\xi\to x}\sigma(t,\xi),
\end{equation} where $\sigma(t,\xi)$ satisfies an associated ordinary differential equation in the time variable. If $P$ satisfies the Hadamard-Petrowsky condition \eqref{LRintroA9}, then $\sigma(t,\cdot)\in\cS'(\rd)$ and we may apply the results of Section \ref{sec:Multipliers}. To meet the results of Section \ref{sec:DecayingSymbols}, we have to assume stronger conditions on $P$ in \eqref{LRintroA8}. Namely, in \cite[Section 5.2.5]{Elena-book} it was observed that
\begin{equation}\label{LRintroA10}
	(\tau,\zeta)\in\bC\times\bC^d, \quad P(i\tau,\zeta)=0 \quad\Rightarrow \quad \Im(\tau)\geq-C(1+|\Im(\zeta)|)^r,
\end{equation}
for some $r\geq1$ and $C>0$, implies \eqref{LRintroA6} for $\mu=1-1/r$, with a new constant $C$. The stronger condition \eqref{LRintroA5}, granting the estimates \eqref{LRintroA7}, is satisfied by the so-called Petrowsky parabolic operators $P$, see \cite{GelShi16} for a detailed study.

Applications of our methods to general classes of operators are challenging, even though we hope the present study will pave the way for future developments.

\paragraph{\bf Structure of the paper.}
The paper is organized as follows. 
In Section~2 we recall the basic tools from time-frequency analysis, Wigner distributions, and operator kernels. 
Section~\ref{sec:Multipliers} provides explicit formulas for the Wigner kernel and the Gabor matrix of a general Fourier multiplier. 
Section~\ref{sec:DecayingSymbols} focuses on multipliers whose symbols belong to Gelfand-Shilov classes, establishing off-diagonal decay for both the Gabor matrix and the Wigner kernel. 
Sections~\ref{sec:ComplHeatEq} and~\ref{sec:waveEq} contain the core applications: Section~\ref{sec:ComplHeatEq} treats the complex heat equation, giving closed-form expressions and sharp Gaussian decay estimates, while Section~\ref{sec:waveEq} addresses the wave equation in dimensions $d \le 3$, including the analysis of lacunas and ghost frequencies. 
Finally, Section~\ref{sec:Metap} develops a metaplectic approach, recovering the previous results and extending them to the complex Hermite equation.


\section{Preliminaries}
Notations in the sequel are standard. In the case there is no ambiguity, we write $y^2$ for $|y|^2$, $y\in\rd$, and $xy$ for the scalar product $x\cdot y$, $x,y\in\rd$. As usual, $\la y\ra=(1+y^2)^{1/2}$. Notations for time-frequency analysis, already used in the Introduction, and for matrix calculus are the following.

Given  $z=(x,\xi)\in\rdd$,  the related  \tfs \,
acting on a function or distribution $f$ on $\rd$ are defined by
\begin{equation}
	\label{eq:kh25}
	\pi (z)f(t) = M_\xi T_xf(t)=e^{2\pi i \xi t} f(t-x), \, \quad t\in\rd.
\end{equation}
We adopt the unitary Fourier transform
\[
\widehat f(\xi)=\int_{\mathbb{R}^d} f(x)\,e^{-2\pi i\,x\cdot \xi}\,dx,\qquad
f(x)=\int_{\mathbb{R}^d} \widehat f(\xi)\,e^{2\pi i\,x\cdot \xi}\,d\xi, \qquad f\in\cS(\rd).
\]

If $\bF$ is either $\bR$ of $\bC$, we denote by $\bF^{d\times d}$ the set of $d\times d$ matrices with coefficients in $\bF$, by $\GL(d,\bF)$ the group of $d\times d$ invertible matrices with coefficients in $\bF$, and by $\Sym(d,\bF)$ the set of $d\times d$ symmetric matrices with coefficients in $\bF$, that is $Q^\top=Q$. If $Q\in\Sym(d,\bR)$, we write $Q\geq0$ (resp, $Q\leq0$) if $Q$ is positive semi-definite (resp. negative semi-definite). For $S_1,S_2\in\bC^{2d\times 2d}$, 
\begin{equation}
	S_j=\begin{pmatrix}
		A_j & B_j\\
		C_j & D_j
	\end{pmatrix}, \qquad A_j,B_j,C_j,D_j\in\bC^{d\times d}, \; j=1,2,
\end{equation}
we define
\begin{equation}\label{deftensmat}
	S_1\otimes S_2:=\left(\begin{array}{cc|cc}
				A_1 & O & B_1 & O\\
				O & A_2 & O & B_2\\
				\hline
				C_1 & O & D_1 & O\\
				O & C_2 & O & D_2
			\end{array}\right),
\end{equation}
where $O=O_d$ is the $d\times d$ null matrix. Similarly, we write $I=I_d$, the $d\times d$ identity matrix. Finally, if $f$ and $g$ are functions on $\rd$, $f\otimes g(x,y)=f(x)g(y)$. 

\subsection{Symplectic analysis}\label{sec:sympAn}
Recall the standard symplectic matrix
	\begin{equation}\label{defJ}
		J=\begin{pmatrix}
			O & I\\
			-I & O
		\end{pmatrix}.
	\end{equation}
	A matrix $S\in\bR^{2d\times2d}$ is {\em symplectic} if $S^\top JS=J$. We denote by $\Sp(d,\bR)$ the group of $2d\times2d$ symplectic matrices. The symplectic group is generated by $J$ and by matrices in the form
	\begin{equation}\label{defDE}
		\cD_E=\begin{pmatrix}
			E^{-1} & O\\
			O & E^\top
		\end{pmatrix},
	\end{equation}
	where $E\in \GL(d,\bR)$ and
	\begin{equation}\label{defVQ}
		V_Q=\begin{pmatrix}
			I& O\\
			Q & I
		\end{pmatrix},
	\end{equation}
	where $Q\in \Sym(d,\bR)$. 
	
	
	The metaplectic group $\Mp(d,\bR)$ is the double cover of $\Sp(d,\bR)$. Moreover, the covering $\pi^{Mp}:\Mp(d,\bR)\to\Sp(d,\bR)$ is a group homomorphism with kernel $\{\pm id_{L^2}\}$. The metaplectic group is generated by the operators 
	\begin{align}
		&\cF f=i^{-d/2}\hat f,\\
		\label{defTE}
		&\mathfrak{T}_Ef(x)=i^{m}|\det(E)|^{1/2}f(Ex), \qquad E\in\GL(d,\bR),\\
		&\mathfrak{p}_Qf(x)=\Phi_Q(x)f(x), \qquad Q\in\Sym(d,\bR),
	\end{align}
	$f\in L^2(\rd)$, where $m=\arg\det(E)$ and $\Phi_Q(x)=e^{i\pi Qx\cdot x}$ is a chirp. With an abuse of notation, for the rest of this work, we will denote $\cF f =\hat f$. 
	
	\noindent

\subsection{Tools from operator theory}

\begin{definition} Consider $f,g\in\lrd$. 
	The cross-Wigner distribution $W(f,g)$ is
	\begin{equation}\label{CWD}
		W(f,g)\phas=\intrd f\Big(x+\frac y2\Big)\overline{g\Big(x-\frac y2\Big)}e^{-2\pi i y\o}\,dy.
	\end{equation} If $f=g$ we write $Wf:=W(f,f)$, the so-called Wigner distribution of $f$. For a thorough study of this distribution we refer to \cite{Gos17}.
\end{definition}
\begin{lemma}[Covariance property; see, e.g., Corollary 217 in \cite{Gos11}]\label{lem:cov}
	Let $\widehat{S} \in \mathrm{Mp}(d, \mathbb{R})$, and let 
	$S = \pi^{Mp}(\widehat{S})$ denote the projection of $\widehat{S}$ on 
	$\mathrm{Sp}(d, \mathbb{R})$. Then, for any $u \in \cS(\mathbb{R}^d)$, 
	the Wigner distribution of $\widehat{S}u$ satisfies
	\[
	W(\widehat{S}u)(x,\xi) 
	= Wu \big( S^{-1}(x,\xi) \big), \quad (x,\xi)\in\rdd.
	\]
\end{lemma}

	The Schwartz kernel of a continuous and linear operator $T:\cS(\rd)\to\cS'(\rd)$ is the (unique) tempered distribution $k_T\in\cS'(\rdd)$ such that:
\begin{equation}\label{Skernel}
	\la Tf,g\ra = \la k_T,g\otimes \bar f\ra, \qquad f,g\in\cS(\rd).
\end{equation}
The Schwartz kernel of a (Kohn-Nirenberg) pseudodifferential operator 
\[
\sigma(x,D)f(x)=\int_{\rd}\hat f(\eta)e^{2\pi ix\cdot\eta}\sigma(x,\eta)d\eta
\]
is given by:
\begin{equation}\label{kerandsymb}
	k_T(x,y)=\cF_2^{-1}\sigma(x,x-y), 
\end{equation}
where $\cF_2$ is the partial Fourier transform with respect to the second $d$ variables:
\begin{equation}\label{defF2}
\cF_2 F(x,\xi)=\intrd F(x,y) e^{-2\pi i y\cdot \xi} dy, \quad F\in\cS(\rdd).
\end{equation}
In particular, if $\sigma(x,D)$ is a Fourier multiplier $$\sigma(x,D)f(x)=\cF^{-1}(\sigma \hf)(x)=(E\ast f)(x),$$ 
where $E=\cF^{-1}\sigma$,
we obtain $k_T(x,y)=\cF^{-1}{\sigma}(x-y)=E(x-y)$.\par
\begin{definition}
The Bessel function $J_0$ of the first kind of order zero, is defined  by the power series
\begin{equation}\label{eq:Bessel}
J_0(z)=\sum_{m=0}^\infty \frac{(-1)^m}{(m!)^2}\left(\frac{z}{2}\right)^{2m}.
\end{equation}
\end{definition}
Equivalently, $J_0$ is defined by the integral representation
\[
J_0(z)=\frac{1}{2\pi}\int_0^{2\pi} e^{iz\cos\theta}\,d\theta.
\]
It satisfies $J_0(0)=1$ and arises naturally as the Fourier transform of the uniform measure on a circle.\par
We will  use the following classical lemma \cite[Appendix A]{Folland}, valid for \(\Re\rho>0\).
\begin{lemma}[Gaussian integral]\label{lem:gauss}
	For \(\rho\in\mathbb{C}\) with \(\Re\rho>0\) and \(c\in\mathbb{C}^d\),
	\begin{equation}\label{eq:gauss}
		\int_{\mathbb{R}^d}\exp\!\big(-2\pi\rho\,|\xi|^2+2\pi\,c\cdot\xi\big)\,d\xi
		=(2\rho)^{-d/2}\,\exp\!\Big(\frac{\pi}{2\rho}\, c\cdot c\Big),
	\end{equation}
	where we denote $ c \cdot c = \sum_{k=1}^d c_k^2,$ $c=(c_1,\dots,c_d)\in\bC^d$.
\end{lemma}

We now exhibit a  characterization of smooth symbols $\sigma\in\cC^\infty(\rdd)$ for a (Weyl) pseudodifferential operator
\begin{equation}
	Op_w(\sigma)f(x)=\int_{\rdd}f(y)\sigma\Big(\frac{x+y}{2},\xi\Big)e^{2\pi i(x-y)\cdot\xi}dyd\xi,\quad f\in\cS(\rd),
\end{equation}
in terms of the continuous Gabor matrix, see Theorem~5.2.10 in \cite{Elena-book}, and \cite{CNRex2015,CNRTAMS2015}.

\begin{definition}
 For $0\leq \mu<\infty$, and $0\leq \nu\leq\infty$, we  write $f \in S^\mu_\nu(\mathbb{R}^{d})$ if $f \in\cC^\infty(\rd)$ and, for every $ \a\in\bN^{d}$,
	\begin{equation}\label{eq:simbsigma}
			\left| \partial^\alpha_\xi f(\xi) \right| 
	\leq C^{|\alpha|+1} (\alpha!)^{\mu} e^{-\varepsilon |\xi|^{1/\nu}},\quad\xi\in\rd,
	\end{equation}
for positive constants $C$ and $\varepsilon$. For $\nu = \infty$ we understand ${1/\nu}=0$ and $\varepsilon = 0$. 
\end{definition}

In what follows, if $v:\rdd\to\bR_{\geq0}$ is a sub-multiplicative weight, i.e., $v(z+w)\leq v(z)v(w)$, $z,w\in\rdd$, we write $m\in\mathcal{M}_v(\rdd)$ if $m:\rdd\to\bR_{\geq0}$ satisfies $m(z+w)\lesssim m(z)v(w)$ for every $z,w\in\rdd$. 

\begin{theorem}\label{C5CGelfandPseudo}  Let $\mu\geq1/2$, and $m\in\mathcal{M}_v(\rdd)$. Assume $g\in S^\mu_\mu(\rd)$ and the following growth condition on the weight $v$:
	\[
	v(z)\lesssim e^{\varepsilon|z|^{1/\mu}},\quad z\in \rdd,
	\]
	for every $\varepsilon>0$. Then the following properties are equivalent for $\sigma\in\cC^\infty(\rdd)$:
	\par {
		(i)} The symbol $\sigma$ satisfies
	\begin{equation}\label{C5simbsmooth} |\partial^\a \sigma(z)|\lesssim m(z) C^{|\a|}(\a!)^{\mu}, \quad \forall\, z\in\rdd,\,\forall \a\in\bN^{2d}.\end{equation}
	{(ii)} There exists $\eps>0$ such that
	\begin{equation}\label{C5unobis2s} |\langle Op_w(\sigma) \pi(z)
		g,\pi(w)g\rangle|\lesssim m\left(\frac{w+z}2\right)e^{-\eps|w-z|^{\frac1\mu}},\qquad \forall\,
		z,w\in\rdd.
	\end{equation}
\end{theorem}
The Theorem above applies to Fourier multipliers $Op_w(\sigma)=\sigma(D_x)$, since  they coincide with their  Weyl form, see \cite[Theorem 5.2.10]{Elena-book} for the proof.
The following estimate will be used in the next section.
\begin{lemma}\label{lem:stime}
If $0 < \nu < \infty$, $\mu\geq 1/2$, for every $\varepsilon_1,\varepsilon_2>0$, there exist $C>0$ and $\varepsilon_3>0$ such that
\begin{equation}\label{eq:stime}
e^{-\varepsilon_1 |\xi + \eta|^{1/\nu}}
e^{-\varepsilon_2 |\xi - \eta|^{1/\mu}}
\leq C\, e^{-\varepsilon_3 (|\xi|^2 + |\eta|^2)^{\rho/2}},\quad \xi,\eta\in\rd,
\end{equation}
	with $\rho = \min\!\left(\tfrac{1}{\nu}, \tfrac{1}{\mu}\right)$.
\end{lemma}
	\begin{proof}	
	Without loss of generality, assume \(\frac{1}{\nu} \leq \frac{1}{\mu}\) (the case \(\frac{1}{\nu} > \frac{1}{\mu}\) is symmetric by swapping \(\nu\) and \(\mu\)). Thus, \(\rho = \frac{1}{\nu}\).
	
\textit{Step 1: Compact Case (Bounded Norms).}
	Let \(r = \sqrt{|\xi|^2 + |\eta|^2}\). For \(r \leq 1\), the left-hand side is continuous and bounded above by 1 (since the exponents are non-positive). The right-hand side is at least \(e^{-\varepsilon_3}\) for some small \(\varepsilon_3 > 0\). Let
	\[
	M = \sup_{r \leq 1} e^{-\varepsilon_1 |\xi + \eta|^{1/\nu} - \varepsilon_2 |\xi - \eta|^{1/\mu}},
	\]
	then choose \(C \geq M e^{\varepsilon_3}\) to ensure the inequality holds for \(r \leq 1\).
	
\textit{Step 2: Large Norms.}
	For \(r > 1\), normalize by setting \(\hat{\xi} = \xi / r\), \(\hat{\eta} = \eta / r\), so \(|\hat{\xi}|^2 + |\hat{\eta}|^{2} = 1\). The left-hand side of \eqref{eq:stime} becomes:
	\[
	e^{-\varepsilon_1 r^{1/\nu} |\hat{\xi} + \hat{\eta}|^{1/\nu} - \varepsilon_2 r^{1/\mu} |\hat{\xi} - \hat{\eta}|^{1/\mu}}.
	\]
	The inequality is equivalent to:
	\[
	-\varepsilon_1 r^{1/\nu} |\hat{\xi} + \hat{\eta}|^{1/\nu} - \varepsilon_2 r^{1/\mu} |\hat{\xi} - \hat{\eta}|^{1/\mu} \leq \log C - \varepsilon_3 r^{\rho} = \log C - \varepsilon_3 r^{1/\nu},
	\]
	or, by rearranging,
	\[
	\varepsilon_1 |\hat{\xi} + \hat{\eta}|^{1/\nu} + \varepsilon_2 r^{1/\mu - 1/\nu} |\hat{\xi} - \hat{\eta}|^{1/\mu} \geq \varepsilon_3 - \frac{\log C}{r^{1/\nu}}.
	\]
	Since \(r > 1\) and \(1/\mu - 1/\nu \geq 0\), \(r^{1/\mu - 1/\nu} \geq 1\). The term \(-\log C / r^{1/\nu}\) is bounded below for \(r > 1\), so we need the left side to be bounded below by a positive constant for large \(r\), adjusted by \(C\).
	
	\textit{Step 3: Critical Directions.}
	The inequality is tightest where the left-hand side exponent is least negative, i.e., where \(\varepsilon_1 |\hat{\xi} + \hat{\eta}|^{1/\nu} + \varepsilon_2 r^{1/\mu - 1/\nu} |\hat{\xi} - \hat{\eta}|^{1/\mu}\) is minimized.
	
	\begin{itemize}
		\item \textit{Critical direction for \(1/\nu\)}: This occurs when \(|\hat{\xi} - \hat{\eta}| = 0\) (so \(\hat{\eta} = \hat{\xi}\), and \(|\hat{\xi}| = |\hat{\eta}| = 1/\sqrt{2}\)). Then \(|\hat{\xi} + \hat{\eta}| = \sqrt{2}\), and the left side becomes:
		\[
		\varepsilon_1 (\sqrt{2})^{1/\nu} + \varepsilon_2 r^{1/\mu - 1/\nu} \cdot 0 = \varepsilon_1 2^{1/(2\nu)}.
		\]
		This is constant, so:
		\[
		\varepsilon_1 2^{1/(2\nu)} \geq \varepsilon_3 - \frac{\log C}{r^{1/\nu}}.
		\]
		We choose \(\varepsilon_3 = \varepsilon_1 2^{1/(2\nu)} / 2 > 0\), and adjust \(C\) for the remainder.
		
		\item \textit{Other direction (\(|\hat{\xi} + \hat{\eta}| = 0\))}: Here \(\hat{\eta} = -\hat{\xi}\), \(|\hat{\xi} - \hat{\eta}| = \sqrt{2}\), and the left side is:
		\[
		\varepsilon_1 \cdot 0 + \varepsilon_2 r^{1/\mu - 1/\nu} (\sqrt{2})^{1/\mu} = \varepsilon_2 2^{1/(2\mu)} r^{1/\mu - 1/\nu}.
		\]
		Since \(1/\mu - 1/\nu \geq 0\), this is at least \(\varepsilon_2 2^{1/(2\mu)}\) (or grows with \(r\)), exceeding \(\varepsilon_3\) for large \(r\).
		
		\item \textit{Intermediate directions}: Both terms are positive, so the sum is larger, making the left-hand side exponent more negative, favoring the inequality.
	\end{itemize}
	
	\textit{Step 4: General Case and Constants.}
	If \(1/\nu > 1/\mu\), swap roles: \(\rho = 1/\mu\), and \(\varepsilon_3 \leq \varepsilon_2 2^{1/(2\mu)}\). In general, it is enough to define:
	\[
	\varepsilon_3 = \frac{1}{2} \min\left( \varepsilon_1 2^{1/(2\nu)}, \varepsilon_2 2^{1/(2\mu)} \right)
	\]
and  choose \(C\) sufficiently large to handle the compact case and any finite \(r\) where the inequality might be tightest.	This completes the proof.
	\end{proof}

\section{Wigner kernel and Gabor matrix of the Fourier Multiplier $\sigma(t,D_x)$}\label{sec:Multipliers}
In this section we will exhibit the explicit expression for the Wigner kernel of the Fourier multiplier \(\sigma(t,D_x)\) defined by
\begin{equation}\label{eq:mult}
		\sigma(t,D_x) f(x) \;=\; (E(t,\cdot) \ast f)(x),\quad f\in\cS(\rd).
\end{equation}
We start with the general definition of Wigner kernel.
\begin{definition}\label{1.1}
	Given a continuous linear operator $T:\,\cS(\rd)\to \cS'(\rd)$, the {\em Wigner operator} associated to $T$ is the linear operator $K:\cS(\rdd)\to \cS'(\rdd)$ such that
	\begin{equation}\label{I3}
		KW(f,g):=W(Tf,Tg),  \qquad f,g\in\cS(\rd).
	\end{equation}
	Its  Schwartz kernel $k_W$ is called the \emph{Wigner kernel} of $T$:
	\begin{equation}\label{I4}
		KW(f,g)(z) = \intrdd k_W(z,w) W(f,g)(w)\,dw,\quad z\in\rdd,\quad f,g\in\cS(\rd).
	\end{equation}
\end{definition}
Such an operator $K$ (resp. Wigner kernel $k_W$) exists and is unique, as proved in Theorem 3.3 of \cite{CRGFIO1}:
\begin{theorem}\label{thmWkernel}
	Let $T:\cS(\rd)\to\cS'(\rd)$ be continuous and linear. Let $k_T\in\cS'(\rdd)$ be its Schwartz kernel, defined as in \eqref{Skernel}. Then, 
	\begin{equation}\label{kkT}
		k_W(x,\xi,y,\eta)=Wk_T(x,y,\xi,-\eta)
	\end{equation}
	is the unique tempered distribution satisfying \eqref{I4}.
\end{theorem}
For any fixed $t\in\bR$, define for short \(T_t = \sigma(t,D_x)\). Then,
$$
\widehat{T_t f}(\xi) = \sigma(t,\xi)\,\widehat{f}(\xi), \quad  \, \sigma(t,\xi)\in\cS'(\rd_\xi),\quad f\in\cS(\rd).$$

\begin{theorem}\label{teo:3.3}
	The Wigner kernel of the Fourier multiplier \(T_t = \sigma(t,D_x)\) is given by
	\begin{equation}\label{eq:wigkernel}
		k_W(x,\xi,y,\eta)
		= \delta(\xi-\eta)\,W(\sigma(t,\cdot))(\xi,y-x)=\delta(\xi-\eta)\,W(E(t,\cdot))(x-y,\xi),
	\end{equation}
	where $\sigma(t,\cdot)\in\cS'(\rd)$ is the multiplier symbol and $E(t,\cdot)\in\cS'(\rd)$ is defined in \eqref{eq:convkernel}.
\end{theorem}
\begin{proof}
	Using 
	$$W(f,g)(x,\xi) =W(\hf,\hg)(\xi,-x)$$ (see, e.g., \cite[Formula (1.84)]{Elena-book}), we can write
	\begin{equation}\label{eq:Wignerft}
		W(f,g)(x,\xi) 
		= \int_{\mathbb{R}^d} 
		\widehat{f}\!\left(\xi+\tfrac{\eta}{2}\right)\,
		\overline{\widehat{g}\!\left(\xi-\tfrac{\eta}{2}\right)}\,
		e^{2\pi i x\cdot \eta}\,d\eta,\quad f,g\in\cS(\rd).
	\end{equation}
	Hence
	\[
	W(T_t f,T_t g)(x,\xi) 
	= \int_{\mathbb{R}^d} 
	\sigma\!\left(t,\xi+\tfrac{\eta}{2}\right)\,
	\overline{\sigma\!\left(t,\xi-\tfrac{\eta}{2}\right)}\,
	\widehat{f}\!\left(\xi+\tfrac{\eta}{2}\right)\,
	\overline{\widehat{g}\!\left(\xi-\tfrac{\eta}{2}\right)}\,
	e^{2\pi i x\cdot \eta}\,d\eta .
	\]
	
	Thus, if
	\[
	M_t(\xi,\eta) 
	= \sigma\!\left(t,\xi+\tfrac{\eta}{2}\right)\,
	\overline{\sigma\!\left(t,\xi-\tfrac{\eta}{2}\right)},
	\]
	then
	\begin{align}
		W(T_t f,T_t g)(x,\xi)&=\cF^{-1}_{\eta\to x}\left( M_t(\xi,\eta)\widehat{f}\!\left(\xi+\tfrac{\eta}{2}\right)\,
		\overline{\widehat{g}\!\left(\xi-\tfrac{\eta}{2}\right)}\right)\notag\\
		&=\cF^{-1}_{\eta\to x}\left( M_t(\xi,\eta)\right)\ast \cF^{-1}_{\eta\to x}\left(\widehat{f}\!\left(\xi+\tfrac{\eta}{2}\right)\,
		\overline{\widehat{g}\!\left(\xi-\tfrac{\eta}{2}\right)}\right)
		\notag	\\
		&=\intrd \left(\intrd e^{2\pi i (x-y)s } M_t(\xi,s) ds\right) W(f,g)(y,\xi) dy\notag\\
		&=\intrd \kappa_t(\xi,x-y) W(f,g)(y,\xi)dy\notag
	\end{align}
	where we set
	\begin{equation}\label{eq:kt}
		\kappa_t(\xi,s) 
		:= \int_{\mathbb{R}^d} e^{2\pi i s\cdot\nu}\,
		M_t(\xi,\nu)\,d\nu .
	\end{equation}
	Let us compute $\kappa_t(\xi,s) $ explicitly. We have
	$$\kappa_t(\xi,s)=\int_{\mathbb{R}^d} e^{2\pi i s\cdot\nu}M_t(\xi,\nu)\,d\nu = \int_{\mathbb{R}^d} e^{2\pi i s\cdot\nu}  \sigma\!\left(t,\xi+\tfrac{\nu}{2}\right)\,
	\overline{\sigma\!\left(t,\xi-\tfrac{\nu}{2}\right)}d\nu=W(\sigma(t,\cdot))(\xi,-s),$$
	by the definition of the Wigner distribution in \eqref{CWD}.
	This proves the first identity in \eqref{eq:wigkernel}. Since $W(\sigma(t,\cdot))(\xi,-s)=W(E(t,\cdot))(s,\xi)$ we obtain the second identity.
	
\end{proof}
\begin{remark}\label{secondproof}
	The Wigner kernel of the Fourier multiplier above could be inferred also from Theorem \ref{thmWkernel}. 
	Namely, using \eqref{kkT}, it boils down to  computing the Wigner distribution of 
	the Schwartz kernel $k_{T,t}(x,y)=E(t,x-y)$. In detail,
	we compute
	\[
	W(k_{T,t})(x,\xi,y,\eta)
	=\int_{\mathbb{R}^{2d}} 
	E\!\left(t,x-\xi+\tfrac{v_1-v_2}{2}\right)\,
	\overline{E\!\left(t,x-\xi-\tfrac{v_1-v_2}{2}\right)}\,
	e^{-2\pi i \big( y\cdot v_1+\eta\cdot v_2\big)}\,dv_1\,dv_2.
	\]
	Set
	\[
	s=({v_1-v_2})/{2}, 
	\qquad 
	l=v_2,
	\qquad 
	v_1=2s+l,\ v_2=l,
	\]
	so that \(dv_1\,dv_2=2^d\,ds\,dl\). Then
	\[
	\begin{aligned}
		W(k_{T,t})(x,\xi,y,\eta)
		&=2^d \int_{\mathbb{R}^{d}}\!\!\int_{\mathbb{R}^{d}}
		E\!\left(t,x-\xi+s\right)\,
		\overline{E\!\left(t,x-\xi-s\right)}\,
		e^{-2\pi i\big(2y\cdot s+(y+\eta)\cdot l\big)}\,ds\,dl \\
		&=2^d \left(\int_{\mathbb{R}^{d}} e^{-2\pi i (y+\eta)\cdot l}\,dl\right)
		\left(\int_{\mathbb{R}^{d}} E\!\left(t,x-\xi+s\right)\,
		\overline{E\!\left(t,x-\xi-s\right)}\,e^{-4\pi i y\cdot s}\,ds\right).
	\end{aligned}
	\]
	The first integral gives a Dirac delta
	\[
	\int_{\mathbb{R}^d} e^{-2\pi i (y-\eta)\cdot l}\,dl = \delta(y+\eta).
	\]
	Hence
	\[
	W(k_{T,t})(x,\xi,y,\eta)
	=2^d\,\delta(y+\eta)\int_{\mathbb{R}^d} 
	E\!\left(t,x-\xi+s\right)\,\overline{E\!\left(t,x-\xi-s\right)}
	\,e^{-4\pi i y\cdot s}\,ds.
	\]
	Now set \(u=2s\), so that \(du=2^d\,ds\). The factor \(2^d\) cancels, and we obtain
	\[
	W(k_{T,t})(x,\xi,y,\eta)
	=\delta(y+\eta)\int_{\mathbb{R}^d} 
	E\!\left(t,x-\xi+\tfrac{u}{2}\right)\,
	\overline{E\!\left(t,x-\xi-\tfrac{u}{2}\right)}\,e^{-2\pi i y\cdot u}\,du.
	\]
	By definition, the integral is exactly the Wigner distribution of \(E(t,\cdot)\):
	\[
	W({E(t,\cdot)})(x-\xi,y)
	=\int_{\mathbb{R}^d} 
	E\!\left(t,x-\xi+\tfrac{u}{2}\right)\,
	\overline{E\!\left(t,x-\xi-\tfrac{u}{2}\right)}\,e^{-2\pi i y\cdot u}\,du.
	\]
	Therefore, the final result is
	\[
	\,W(k_{T,t})(x,\xi,y,\eta) = \delta(y+\eta)\,W(E(t,\cdot))(x-\xi,y)\, .
	\]
	Using \eqref{kkT}, we obtain
	$$k_W(x,\xi,y,\eta)=W(k_{T,t})(x,y,\xi,-\eta)=\delta(\xi-\eta)\,
	W(E(t,\cdot))(x-y,\xi).$$
\end{remark}

In the next sections, we shall compute also the Gabor matrix of the Fourier multiplier $\sigma(t,D_x)$ in \eqref{eq:mult}, with $\sigma\in\cS'(\rd)$, where the Gabor matrix is defined in \eqref{defGMat} in the Introduction. We fix now the window $g(t)=e^{-\pi |t|^2}$, $t\in\rd$, so that $h\in\cS'(\bR^{4d})\cap\mathcal{C}^\infty(\bR^{4d})$. From \eqref{convgg1} we can write
\begin{equation}\label{LReqS2}
	|h(z,w)|^2=(k_W\ast\Phi)(z,w), \qquad z=(x,\xi),w=(y,\eta)\in\rdd,
\end{equation}
where
\begin{equation}\label{LReqS3}
	\Phi(z,w)=(Wg\otimes Wg)(z,w)=2^d e^{-2\pi(|z|^2+|w|^2)}.
\end{equation}
Combining \eqref{LReqS2} with Theorem \ref{thmWkernel} and using the symmetric properties of $\Phi$, we deduce
\begin{equation}\label{LReqS4}
	|h(z,w)|^2=(Wk_T\ast\Phi)(x,y,\xi,-\eta),
\end{equation}
where as before $k_T\in\cS'(\rdd)$ is the Schwartz kernel of $T$. 
In the sequel we shall use the following result to compute $|h(z,w)|$.
\begin{proposition}\label{LRpropS1}
	Let $T:\cS(\rd)\to\cS'(\rd)$ be a linear and continuous map, $k_T$ be its Schwartz kernel, then
	\begin{equation}\label{LReqS5}
		|h(z,w)|=|V_\Psi k_T(x,y,\xi,-\eta)|, \qquad z=(x,\xi),w=(y,\eta)\in\rdd,
	\end{equation}
	where $V_\Psi$ is the short-time Fourier transform with the window
	\begin{equation}\label{LReqS6}
		\Psi(x,y)=e^{-\pi(x^2+y^2)}.
	\end{equation}
	Assume further $T=\sigma(D)$ with $\sigma\in\cS'(\rd)$, $E=\cF^{-1}\sigma\in\cS'(\rd)$, then
	\begin{equation}\label{LReqS7}
		|h(z,w)|=\left| \int_{\rd}H(x,y,\xi,\eta,r)E(r)dt \right|,
	\end{equation}
	where $H\in\cS(\rd_r)$ with
	\begin{equation}\label{LReqS8}
		|H(x,y,\xi,\eta,r)|=2^{-d/2}\exp\left(-\frac\pi 2 \left((\xi-\eta)^2+(x-y-r)^2\right)\right).
	\end{equation}
\end{proposition}
In the proof, we shall use the following well-known identities, see \cite{Elena-book}.
\begin{lemma}\label{LRlemmaS1}
	With $\Phi$ as in \eqref{LReqS3} and $\Psi$ as in \eqref{LReqS6} we have
	\begin{equation}
		Wf\ast\Phi=|V_\Psi f|^2
	\end{equation}
	for every $f\in\cS'$.
\end{lemma}
\begin{lemma}\label{LRlemmaS2}
	Taking as before $g(t)=e^{-\pi|t|^2}$, $t\in\rd$, we have
	\begin{equation}\label{LReqS10}
		|\la \pi(z)g,\pi(w)g\ra|=2^{-d/2}\exp\left(-\frac\pi 2(w-z)^2\right), \qquad z,w\in\rdd.
	\end{equation}
\end{lemma}
\begin{proof}[Proof of Proposition \ref{LRpropS1}]
	From \eqref{LReqS4}, taking $f=k_T\in\cS'(\rdd)$ in Lemma \ref{LRlemmaS1}, we have
	\begin{equation}
		|h(z,w)|^2=|Wk_T\ast\Phi(x,y,\xi,-\eta)|=|V_\Psi k_T(x,y,\xi,-\eta)|^2
	\end{equation}
	and \eqref{LReqS5} is proven. Assuming further $T=\sigma(D)$ we have $k_T(x,y)=E(x-y)$, hence
	\begin{align}
		|V_\Psi k_T(x,y,\xi,-\eta)|=\left|
			\int_{\rdd}\exp\big(-\pi[ (\lambda-x)^2+(\mu-y)^2]\big)E(\lambda-\mu)e^{-2\pi i(\lambda\xi-\mu\eta)}d\lambda d\mu
		\right|.
	\end{align}
	Setting $r=\lambda-\mu$, $d\lambda =dr$, and using \eqref{LReqS5}, we obtain
	\begin{equation}
		|h(z,w)|=\left|\int_{\rd}H(x,y,\xi,\eta,r)E(r)dr\right|,
	\end{equation}
	where, for $g(t)=e^{-\pi |t|^2}$, $t\in\rd$, as before,
	\begin{equation}
		H(x,y,\xi,\eta,r)=e^{-2\pi ir\xi}\la \pi(y,\eta)g,\pi(x-r,\xi)g\ra.
	\end{equation}
	By applying Lemma \ref{LRlemmaS2}, we then obtain \eqref{LReqS8}, thereby concluding the proof.
	
\end{proof}
\begin{remark}\label{rem38}
	Under more precise information on $E$, we may deduce estimates for $h(z,w)$. In particular, if $E(r)dr$ is a positive Borel measure, then from \eqref{LReqS7} and \eqref{LReqS8}, we obtain
	\begin{align*}
		|h(z,w)|&\leq\int_{\rd}|H(x,y,\xi,\eta,r)|E(r)dr\leq 2^{-d/2}\exp\left(-\frac\pi 2  (\xi-\eta)^2\right)\int_{\rd}\exp\left(-\frac\pi 2 (x-y-r)^2 \right)E(r)dr\\
		&=2^{-d/2}\exp\left(-\frac\pi 2  (\xi-\eta)^2\right) \left(\exp\Big(-\frac\pi 2(\cdot)^2\Big)\ast E\right)(x-y).
	\end{align*}
\end{remark}

\section{Fourier multipliers with symbols in Gelfand-Shilov classes}\label{sec:DecayingSymbols}

Omitting for simplicity $t$-dependence we shall now consider \eqref{eq:wigkernel} under more precise information on the symbol $\sigma$, in the framework of the Gelfand-Shilov classes $S^\mu_\nu(\mathbb{R}^{2d})$ defined in \eqref{eq:simbsigma}. 
	As a preliminary step, assume
	\begin{equation}
	\label{LReqK1}
		|\sigma(\xi)|\leq C, \qquad \xi\in\rd,
	\end{equation}
	for some $C>0$, that is $\sigma\in L^\infty(\rd)$. Then, from \eqref{eq:wigkernel} we have
	\begin{equation}
	\label{LReqK2}
		k_W(z,w)=\delta(\xi-\eta)a(x-y,\xi), \qquad z=(x,\xi), \, w=(y,\eta),
	\end{equation}
	where
	\begin{equation}
	\label{LReqK3}
		a(s,\xi)=W\sigma(\xi,s)=\cF^{-1}_{r\to s}b(\xi,s),
	\end{equation}
	with
	\begin{equation}
		b(\xi,r)=\sigma(\xi+r/2)\overline{\sigma(\xi-r/2)},
	\label{LReqK4}
	\end{equation}
	satisfying
	\begin{equation}
	\label{LReqK5}
		|b(\xi,r)|\leq C^2, \qquad \xi,r\in\rd.
	\end{equation}
	From \eqref{LReqK3} we have that $a(x-y,\xi)$ is the kernel of a Fourier multiplier acting on the $x$-variables, depending on the parameter $\xi\in\rd$, with $L^2$-norm bounded uniformly by $C^2$, in view of \eqref{LReqK5}. 
	Approaching further the case of the propagators of heat and wave equations, 
	we may consider the symbols $\sigma(\xi)$ satisfying for every $\alpha\in\mathbb{N}^d$
	\[
	|\partial^\alpha \sigma(\xi)| \leq C_\alpha, \qquad \xi \in \mathbb{R}^d.
	\]
	By applying Theorem~1.3 and Theorem~1.4 in~\cite{CRPartI}, 
	we have that for every integer $N \geq 0$, 
	the kernel 
	\eqref{LReqK3}
	can be written as
	\[
	a(x - y, \xi) = \langle x - y \rangle^{-N} \, a_N(x - y, \xi),
	\]
	where $a_N(x - y, \xi)$ is again the kernel of an $L^2$-bounded convolution operator 	in the $x$-variables.

	We want to give more precise information on $a(x-y,\xi)$ for symbols $\sigma$ as in \eqref{eq:simbsigma}. We first limit our attention to the relevant cases $(\mu,\nu)=(1/2,1/2)$, $(\mu,\nu)=(0,\infty)$. 
	\begin{theorem}\label{LRthm1circ}
		Let $\sigma$ satisfy \eqref{eq:simbsigma} with $\mu=\nu=1/2$, i.e., $\sigma\in S^{1/2}_{1/2}(\rd)$. Then, the Wigner kernel of $\sigma(D_x)$ is given by
		\begin{equation}
			\label{LReqK6}
			k_W(z,w)=\delta(\xi-\eta)a(x-y,\xi), \qquad z=(x,\xi),\, w=(y,\eta)\in\rdd,
		\end{equation}
		where $a\in\mathcal{C}^\infty(\rdd)$ satisfies
		\begin{equation}
	\label{LReqK7}
		|a(s,\xi)|\leq C e^{-\varepsilon|\xi|^2}e^{-\varepsilon|s|^2},
	\end{equation}
	for some $C>0$ and $\varepsilon>0$.
	\end{theorem}
	\begin{proof}
		According to \eqref{LReqK3},
		\begin{equation}
			a(s,\xi)=W\sigma(\xi,s).
		\end{equation}
		Since $\sigma\in S^{1/2}_{1/2}(\rd)$, then $W\sigma\in S^{1/2}_{1/2}(\rdd)$, see for example \cite{T1}. In particular, for some $\varepsilon>0$ and $C>0$,
		\begin{equation}
			|W\sigma(\xi,s)|\leq Ce^{-\varepsilon(|\xi|^2+|s|^2)}, \qquad \xi\in\rd, \, s\in\rd.
		\end{equation}
		Hence, we obtain \eqref{LReqK7}.
		
	\end{proof}
	
	\begin{remark}\label{rem:4.2}
		In view of the product by $\delta(\xi-\eta)$ we may re-write \eqref{LReqK6} and \eqref{LReqK7} as
		\begin{equation}
			k_W(z,w)=\delta(\xi-\eta)\tilde a(x,\xi,y,\eta),\qquad  z=(x,\xi), \, w=(y,\eta),
		\end{equation}
		with
		\begin{equation}
			|\tilde a(x,\xi,y,\eta)|\leq Ce^{-\frac\varepsilon 2 (|\xi|^2+|\eta|^2)}e^{-\varepsilon|x-y|^2}.
		\end{equation}
	\end{remark}
	
	\begin{theorem}\label{LRthm2circ}
		Assume $\mu=0$ and $\nu=\infty$ in \eqref{eq:simbsigma}, that is for every $\alpha\in\mathbb{N}^d$
		\begin{equation}
		\label{LReqK8}
			|\partial_\xi^\alpha \sigma(\xi)|\leq C^{|\alpha|+1}, \qquad \xi\in\rd,
		\end{equation}
		for some constant $C>0$. Then, the Wigner kernel of $\sigma(D_x)$ is given by
		\begin{equation}
		\label{LReqK9}
			k_W(z,w)=\delta(\xi-\eta)\chi_{\{|x-y|\leq T\}}a(x-y,\xi),
		\end{equation}
		where $\chi_{\{|x-y|\leq T\}}$ is the characteristic function of the region $\{(x,y)\in\rdd:|x-y|\leq T\}$, with $T>0$ depending on $C$ in \eqref{LReqK8} and $a(x-y,\xi)$ defined as before, cf. \eqref{LReqK2} and \eqref{LReqK3}.
	\end{theorem}
	\begin{proof}
		In view of \eqref{LReqK3}, the identity \eqref{LReqK9} amounts to prove that
		\begin{equation}
			a(s,\xi)=\cF^{-1}_{r\to s}b(\xi,s),
		\end{equation}
		with
		\begin{equation}
			b(\xi,r)=\sigma(\xi+r/2)\overline{\sigma(\xi-r/2)}
		\end{equation}
		supported in the region $|s|\leq T$ for a constant $T$ independent of $\xi$. To this end, we begin by observing that under the assumption \eqref{LReqK8}, for a new constant $C>0$,
		\begin{equation}
		\label{LReqK10}
			|\partial^\alpha_rb(\xi,r)|\leq C^{|\alpha|+1},\qquad \xi,r\in\rd.
		\end{equation}
		Hence, $b(\xi,r)$ extends to an entire function with respect to $r$
		\begin{equation}
			F(\xi,\zeta)=\sum_\beta \frac{\partial^\beta b(\xi,0)}{\beta!}\zeta^\beta, \qquad \xi\in\rd, \, \zeta=r+iv\in\bC^d.
		\end{equation}
		We now evaluate $F(\xi,r+iv)$ by Taylor expanding at the point $r\in\rd$:
		\begin{equation}
			F(\xi,r+iv)=\sum_\beta \frac{\partial^\beta b(\xi,r)}{\beta!}(iv)^\beta
		\end{equation}
		and from \eqref{LReqK10} it follows
		\begin{equation}
		\label{LReqK11}
			|F(\xi,r+iv)|\leq C_1 e^{C_2|v|},
		\end{equation}
		with constants $C_1,C_2$ depending only on $C$ in \eqref{LReqK10}. We then apply the Theorem of Paley-Wiener-Schwartz, see \cite[Theorem 7.3.1]{Hormander1}. From \eqref{LReqK11}, we have that $F(\xi,\zeta)$ is the Fourier-Laplace transform of a distribution with compact support in $\{|s|\leq T\}$, with $T$ depending on $C_2$ in \eqref{LReqK11}, hence $T$ is independent of the parameter $\xi$. This concludes the proof.
		
	\end{proof}
	
	About possible extensions of Theorems \ref{LRthm1circ} and \ref{LRthm2circ} to symbols $\sigma(\xi)$ satisfying the general estimates \eqref{eq:simbsigma} we note that a version of Theorem \ref{LRthm1circ} can be easily proven under the assumption $\mu=\nu>1/2$ by using Teofanov's results in \cite{T1}. We provide here the following counterpart of Theorem \ref{LRthm2circ} in the case $\mu>0$.
	
	\begin{theorem}\label{LRthm3circ}	
		Let $\sigma$ satisfy \eqref{eq:simbsigma} with $\mu>0$, $\nu=\infty$, i.e.,
		\begin{equation}
		\label{LReqK12}
			|\partial_\xi^\alpha(\xi)|\leq C^{|\alpha|+1}(\alpha!)^\mu, \qquad \xi\in\rd.
		\end{equation}
		Then, the Wigner kernel of $\sigma(D_x)$ is given by
		\begin{equation}
		\label{LReqK13}
			k_W(z,w)=\delta(\xi-\eta)a_1(x-y)a_2(x-y,\xi), \qquad z=(x,\xi),\, w=(y,\eta),
		\end{equation}
		where $a_1\in\mathcal{C}^\infty(\rd)$ satisfies
		\begin{equation}
		\label{LReqK14}
			|a_1(s)|\leq Ce^{-\varepsilon|s|^{1/\mu}}
		\end{equation}
		for some $C>0$, $\varepsilon>0$, and $a_2(x-y,\xi)$ is the kernel of a bounded convolution operator
		\begin{equation}
		\label{LReqK15}
			\Big\Vert\int_{\rd}a_2(x-y,\xi)g(y)dy\Big\Vert_2\leq C\norm{g}_2, \qquad g\in L^2(\rd),
		\end{equation}
		with $C$ independent of $\xi$. We address the proof below for the expression of $a_2$.
	\end{theorem}
	\begin{proof}
		Using again \eqref{eq:wigkernel} we write
		\begin{equation}
		 k_W(z,w)=\delta(\xi-\eta)\cF^{-1}_{r\to x-y}b(\xi,x-y),
		\end{equation}
		where now, for a new constant $C>0$,
		\begin{equation}
		\label{LReqK16}
			|\partial^\alpha_\xi\partial^\beta_r b(\xi,r)|\leq C^{|\alpha|+|\beta|+1}(\alpha!\beta!)^\mu,
		\end{equation}
		in view of standard estimates on factorials, from \eqref{LReqK12}. Integrating by parts, we have for $N\geq0$,
		\begin{equation}
		\label{LReqK17}
			|x-y|^{2N}a(x-y,\xi)=\int_{\rd}e^{2\pi i(x-y)r}b_N(\xi,r)dr,
		\end{equation}
		with
		\begin{equation}
		\label{LReqK18}
			b_N(\xi,r)=\frac{1}{(2\pi)^{2Nd}}\Delta^N_rb(\xi,r).
		\end{equation}
		In view of \eqref{LReqK16},
		\begin{equation}
		\label{LReqK19}
			|\partial_\xi^\alpha\partial_r^\beta b_N(\xi,r)|\leq C^{2N+|\alpha|+|\beta|+1}((2N)!\alpha!\beta!)^\mu.
		\end{equation}
		Define now for $\varepsilon>0$
		\begin{equation}
		\label{LReqK20}
			h(s)=\sum_{N=0}^\infty \varepsilon^{2N}\frac{|s|^{2N}}{((2N)!)^\mu}, \qquad s\in\rd.
		\end{equation}
		We have $h\in\mathcal{C}^\infty(\rd)$ and for a new constant $\varepsilon>0$,
		\begin{equation}
			h(s)\sim e^{\varepsilon |s|^{1/\mu}} \qquad (|s|\to\infty),
		\end{equation}
		here $\sim$ denotes the usual asymptotic equivalence, see for example the proof of Proposition 5.2.4 in \cite{Elena-book}. We set
		\begin{equation}
			a_1(s)=\frac{1}{h(s)}\sim e^{-\varepsilon|s|^{1/\mu}}.
		\end{equation}
		We have from \eqref{LReqK17}, \eqref{LReqK18} and \eqref{LReqK20},
		\begin{equation}
			h(x-y)a(x-y,\xi)=\int_{\rd} e^{2\pi i(x-y)r}\tilde b(\xi,r)dr
		\end{equation}
		with
		\begin{equation}
			\tilde b(\xi,r)=\sum_{n=0}^\infty\frac{\varepsilon^{2N}}{((2N)!)^\mu}b_N(\xi,r).
		\end{equation}
		If $\varepsilon>0$ is chosen sufficiently small, the series is uniformly convergent and $\tilde b\in L^\infty(\rd_\xi\times\rd_r)$. More precisely, for all $\alpha,\beta$ we have from \eqref{LReqK19},
		\begin{equation}
			|\partial_x^\alpha\partial_r^\beta\tilde b(\xi,r)|\leq C^{|\alpha|+|\beta|+1}(\alpha!\beta!)^\mu.
		\end{equation}
		To sum up,
		\begin{equation}
			a(x-y,\xi)=a_1(s)a_2(x-y,\xi),
		\end{equation}
		with
		\begin{equation}
			a_2(x-y,\xi)=\int_{\rd}e^{2\pi i(x-y)r}\tilde b(\xi,r)dr
		\end{equation}
		kernel of a bounded convolution operator, as required.
	\end{proof}
	
	\medskip
	
Let us pass to the study of the Gabor matrix for Gelfand-Shilov symbols. 
The following result is an easy application of Lemma \ref{lem:stime}.
	\begin{theorem}\label{teor:Luigi} Consider the Gabor matrix of the multiplier  $\sigma(D_x)$, with symbol $\sigma$ satisfying \eqref{eq:simbsigma}, i.e., $\sigma\in S^{\mu}_{\nu}(\rd)$, with $\mu,\nu\geq1/2$,  and window $g\in S^{1/2}_{1/2}(\rd)$.
		For new constants $C$ and $\varepsilon$, if $\nu = \infty$ we have the estimate
	\begin{equation}\label{Luigi1}
		\left| \langle \sigma(D_x)\, \pi(z)g,\, \pi(w)g \rangle \right|
		\leq C\, e^{-\varepsilon |w - z|^{1/\mu}},
	\end{equation}
		with $z=(x,\xi)$, $w=(y,\eta)\in\rdd$.
		If $1/2\leq\nu < \infty$, we obtain
	\begin{equation}\label{Luigi2}
		\left| \langle \sigma(D_x)\, \pi(z)g,\, \pi(w)g \rangle \right|
		\leq C\, e^{-\varepsilon (|\xi|^{2} + |\eta|^{2})^{\rho/2}}
		e^{-\varepsilon |x - y|^{1/\mu}},
	\end{equation}
		with $\rho = \min\!\left(\tfrac{1}{\nu}, \tfrac{1}{\mu}\right)$.
	\end{theorem}
	
	\begin{proof}
		The first part follows from formula \eqref{C5unobis2s}
		with weight $m = \text{const}$.
		The second follows by fixing in \eqref{C5unobis2s} the weight 
		$m (\xi)= e^{-\varepsilon_1 |\xi|^{1/\nu}}$ and using Lemma \ref{lem:stime}. 
	\end{proof}
	
	\medskip
	Relevant model examples for the preceding results are the $t$-dependent symbols $\sigma(t,\cdot)\in S^\mu_\nu(\rd)$ of the propagators \eqref{LRintroA3} and \eqref{LRintroA4}, namely:
	\begin{itemize}
		\item [(i)]	The wave equation: $\sigma(\xi)$ is given by
			\begin{equation} \label{eq:sigma-onde}
				\sigma(t,\xi) = \frac{\sin(2\pi |\xi| t)}{2\pi |\xi|}, \qquad t\in\bR,\, \xi \in \mathbb{R}^d\setminus\{0\},
			\end{equation} with
		$\mu = 0$, $\nu = \infty$. In fact, for a constant $C>0$ depending on $t$,  
		\[
			|\partial^\alpha_\xi \sigma(\xi)| \le C^{|\alpha|+1}, \qquad \xi \in \mathbb{R}^{d}.
		\]
		\item[(ii)] The complex heat equation:  
		\[
			\sigma(t,\xi)=e^{-4\pi^2\gamma t \xi^2}, \qquad t>0,\, \xi\in\rd,
		\]
		with $\gamma \in \mathbb{C}$, $\Re \gamma > 0$, $\mu = \tfrac{1}{2}$, $\nu = \tfrac{1}{2}$.
	\end{itemize}  

 In fact, the  Gaussian is in the Gelfand-Shilov class $S^{1/2}_{1/2}(\rd)$, cf., e.g., \cite{T2,ToftGS}.

Finally, we compare the estimates in Theorem \ref{LRthm1circ}, 
Theorem \ref{LRthm2circ}, 
Theorem~\ref{LRthm3circ} 
for the Wigner kernel with those in Theorem~\ref{teor:Luigi} for the Gabor matrix. Generally speaking, the asymptotic off-diagonal decay of the Wigner kernel is stronger, 
at the expense of a loss of smoothness, due to the presence in~\eqref{LReqK2} 
of the factor $\delta(\xi - \eta)$ and the possible local singularities 
of the convolution term $a(x - y, \xi)$. 

More precisely, in the case $(\mu, \nu) = (0, \infty)$, in view of~ \eqref{LReqK9}, 
the Wigner kernel is compactly supported in a neighborhood of the origin 
for the variables $r = \xi - \eta$ and $s = x - y$. 
Whereas in \eqref{Luigi1}, for the Gabor matrix we cannot improve the exponential estimates 
with $\mu = 1/2$. In fact, the presence of the window $g \in S^{1/2}_{1/2}(\mathbb{R}^d)$ 
provides smoothness, but gives such a threshold to the decay.

Similarly, for the case $0 < \mu < 1/2$, $\nu = \infty$, 
the decay given in Theorem~\ref{LRthm3circ} is stronger than that in~\eqref{Luigi1}. 
For $\mu \ge 1/2$, $\nu = \infty$, the estimates for the Wigner kernel 
and Gabor matrix essentially coincide, as well as in the case $(\nu, \mu) = (1/2, 1/2)$ 
treated in Remark~\ref{rem:4.2} and~\eqref{Luigi2}. 

Explicit computations for the wave and complex heat equations follow in the next sections.

\subsection{Applications to Cauchy problems for operators with constant coefficients}
As an introduction to the computations in the next sections for heat and wave equations, we recall the results in \cite{CNRTAMS2015,Elena-book} concerning the Cauchy problem for general linear partial differential operators with constant coefficients $P(\partial,D_x)$ of the form \eqref{LRintroA8}, \eqref{LRintroA9} and corresponding propagators $\sigma(t,D_x)$ in \eqref{eq:convkernel}, see \cite{Rauch}. We argue under the strong Hadamard-Petrowsky condition \eqref{LRintroA10}, namely for $C>0$, $r\geq1$
\begin{equation}
	(\tau,\zeta)\in\bC\times\bC^d, \, P(i\tau,\zeta)=0 \quad \Longrightarrow \quad \Im(\tau)\geq -C(1+|\Im(\zeta)|)^r.
\end{equation}
This condition is satisfied by the wave equation and, as we shall verify in the next section, by the complex heat equation. Concerning the Gabor matrix, we have:
\begin{theorem}\label{C5teo5.2}
	Assume $P$ satisfies \eqref{LRintroA10} for some $C>0$, $r\geq1$, and set  $\mu=\max\{1/2,1-1/r\}$. If $g\in S^{1/2}_{1/2}(\rd)$ then  $\sigma(t,D_x)$   satisfies 
	\begin{equation}\label{C5hpm3} |\langle \sigma(t,D_x) \pi(z)
		g,\pi(w)g\rangle|\leq C e^{-\eps |w-z|^{\frac1\mu}},\qquad \,
		z,w\in\rdd,
	\end{equation}
	for some $\epsilon>0$ and for a new constant $C>0$. The constants $\epsilon$ and $C$ are uniform when $t$ lies in bounded subsets of $[0,+\infty)$.
\end{theorem}
Let us show how Theorem \ref{C5teo5.2} can be deduced from the preceding Theorem \ref{teor:Luigi}. We need the following result, see \cite[Theorem 5.2.19]{Elena-book}.
\begin{theorem}\label{LRthmZZ}
	Assume $P$ satisfies \eqref{LRintroA10} for some $C>0$, $r\geq1$. Then, the symbol $\sigma(t,\xi)$ of the corresponding propagator satisfies the estimates
	\begin{equation}
		|\partial_\xi^\alpha\sigma(t,\xi)|\leq C^{(t+1)|\alpha|+t}(\alpha!)^\mu, \quad t\geq0,\, \alpha\in\bN^d
	\end{equation}
	with $\mu=1-1/r$, for a new constant $C>0$.
\end{theorem}
In view of Theorem \ref{LRthmZZ}, under the assumptions of Theorem \ref{C5teo5.2} we have $\sigma(t,\cdot)\in S^\mu_\infty(\rd)$ with $\mu\geq1/2$, and then \eqref{C5hpm3} follows from the first part of Theorem \ref{teor:Luigi}. For the complex heat equation, the estimates \eqref{C5hpm3} will be improved in the next section, see \eqref{LRintroA7} and the second part of Theorem \ref{teor:Luigi}.

The preceding results give also information on the Wigner kernel of the propagators $\sigma(t,D_x)$ in Theorem \ref{C5teo5.2}. Namely, since $\sigma(t,\cdot)\in S^\mu_\infty(\rd)$ with $\mu=1-1/r$, $r\geq1$, in view of Theorem \ref{LRthmZZ}, we may apply Theorems \ref{LRthm2circ} and \ref{LRthm3circ} to the corresponding $k_W$. In the case $\mu>0$, the estimate for the term $a_1$ in \eqref{LReqK14} keeps valid with dependence on $t$:
\begin{equation}
	|a_1(s)|\leq Ce^{-\varepsilon |s|^{1/\mu}}, \qquad s=x-y,
\end{equation}where now the limitation $\mu\geq 1/2$ disappears.

Theorem \ref{LRthm2circ} gives an interesting application to hyperbolic problems. We recall that the operator $P(\partial_t,D_x)$ is called hyperbolic with respect to $t$ if the direction $N=(1,0,\ldots,0)\in\bR\times\rd$ is noncharacteristic for $P$, i.e., its principal symbol does not vanish at $N$, and $P$ satisfies \eqref{LRintroA9}. Note that $P$ is not required to be strictly hyperbolic, namely the roots of the principal symbol are allowed to coincide.
\begin{proposition}\label{LRpropWW}
	Assume that $P(\partial_t,D_x)$ is hyperbolic with respect to $t$. Then, the condition \eqref{LRintroA10} is satisfied with $r=1$ for some $C>0$.
\end{proposition}
For the proof, see for example \cite[Proposition 5.2.21]{Elena-book}. Relevant example is given by the wave equation. It follows from Proposition \ref{LRpropWW} that for hyperbolic operators $\mu=1-1/r=0$ and from Theorem \ref{LRthmZZ} we have that the symbol $\sigma(t,\xi)$ of the corresponding propagator belongs to $S^0_\infty(\rd)$ for any fixed $t\in\bR$. Hence, we conclude that $k_W(x,\xi,y,\eta)$ is supported in a strip $|x-y|\leq T$ with $T$ depending on $t$. A much more precise analysis will be given in Section \ref{sec:waveEq} for the wave equation, by including estimates in the possible lacunas of the region $|x-y|\leq T$.

\section{The complex heat equation}\label{sec:ComplHeatEq}
This section will contain the thorough study of the complex heat operator, showing the properties of the related multiplier and, consequently, the related  Gabor matrix decay.  Then, the Wigner  distribution of the solution will be exhibited, as well as its Wigner kernel. Finally, the explicit computation for its Gabor matrix will be carried out, showing the exact constant appearing in the decay estimates. 

Namely,  consider the initial value problem:
\begin{equation}\label{e0}
\partial_t u = (\alpha + i\beta) \Delta u, \quad u(0, x) = u_0(x),
\end{equation}
where $\alpha+i \beta \in \mathbb{\bC}\setminus\{0\}$,  $\alpha\geq 0$, are fixed constants.
Let us define\begin{equation}\label{eq:gamma}
\gamma:=\alpha+i\beta\in\mathbb{\bC}\setminus\{0\}.
\end{equation}

Applying the Fourier transform in space to both sides of \eqref{e0} yields:
\[
\partial_t \hat{u}(t, \xi) = -4\pi^2 (\alpha + i\beta) |\xi|^2 \hat{u}(t, \xi),
\]
with initial condition $\hat{u}(0, \xi) = \hat{u}_0(\xi)$. Solving this ODE gives:
\begin{equation}\label{ee0}
	\hat{u}(t, \xi) = e^{-4\pi^2 (\alpha + i\beta) t |\xi|^2} \hat{u}_0(\xi).
\end{equation}

Applying the inverse Fourier transform:
\begin{equation}\label{e1}
	u(t, x) = \mathscr{F}^{-1} \left( e^{-4\pi^2 (\alpha + i\beta) t |\xi|^2} \hat{u}_0(\xi) \right)(x).
\end{equation}
The inverse Fourier transform of the exponential factor is a complex Gaussian kernel. Therefore, the solution can be written explicitly as:
\begin{equation}\label{e2}
	u(t,x) = \frac{1}{(4\pi (\alpha + i\beta)t)^{d/2}} \int_{\mathbb{R}^d} \exp\left( -\frac{|x - y|^2}{4(\alpha + i\beta)t} \right) u_0(y) \, dy.
\end{equation}

This represents the convolution of the initial data $u_0(x)$ with a complex-valued heat kernel:
\[
u(t,x) = (K_t * u_0)(x),
\]
where
\begin{equation}\label{eq:kt}
	K_t(x)= \frac{1}{(4\pi (\alpha + i\beta)t)^{d/2}}   \exp\left( -\frac{|x|^2}{4(\alpha + i\beta)t} \right).
\end{equation}
When \( \alpha =  1 \) and $\beta=0$, we get the heat equation, while $\alpha=0$ and  \( \beta = 1 \) corresponds to the Schr\"odinger equation.
%

Let us define the propagator
\begin{equation}\label{eq:evolheat}
	\sigma(t,D_x)  f:=e^{-4\pi^2 \gamma t|D|^2} f=K_t \ast f.
\end{equation}
A first step consists in inferring decay properties of the related Gabor matrix by means of the existing results in the literature:
\begin{proposition}\label{propLuigi}
	The equation \eqref{e0} with $\alpha>0$ satisfies \eqref{LRintroA10} with $r=2$,
	hence the Gabor matrix of $\sigma(t,D_x)$
	satisfies \eqref{C5hpm3} in Theorem \ref{C5teo5.2} with \(1/\mu=2\). In detail,
		\begin{equation}\label{C5hpm3heat} |\langle \sigma(t,D_x) \pi(z)
		g,\pi(w)g\rangle|\leq C e^{-\eps |w-z|^{2}},\qquad \,
		z,w\in\rdd.
	\end{equation}
\end{proposition}

\begin{proof}
	By writing $\zeta=\xi+i\eta$ with $\xi,\eta\in\mathbb{R}^d$, the dispersion relation in \eqref{LRintroA10},
	\[
	P(i\tau,\zeta)=0
	\]
	gives
	\[
	\tau = i(\alpha+i\beta)|\zeta|^2
	= i(\alpha+i\beta)\big(|\xi|^2-|\eta|^2+2i\,\xi\cdot \eta\big).
	\]
	Taking the imaginary part,
	\[
	-\Im \tau = -\alpha|\xi|^2 + \alpha|\eta|^2 + 2\beta\, \xi\cdot\eta .
	\]
	
	By Young's inequality, 
	\[
	2|\xi\cdot \eta| \leq \varepsilon|\xi|^2 + \varepsilon^{-1}|\eta|^2, \qquad \varepsilon>0.
	\]
	Therefore
	\[
	-\Im \tau \leq (-\alpha+|\beta|\varepsilon)|\xi|^2 
	+ \Big(\alpha+|\beta|\varepsilon^{-1}\Big)|\eta|^2.
	\]
	Choosing $\varepsilon=\alpha/|\beta|$ (any $\varepsilon>0$ if $\beta=0$), the coefficient of $|\xi|^2$ vanishes, and we obtain
	\[
	-\Im \tau \leq \Big(\alpha+\tfrac{\beta^2}{\alpha}\Big)\,|\eta|^2 
	= C\,|\Im\zeta|^2.
	\]
	This shows that condition \eqref{LRintroA10} holds with $r=2$, uniformly in $d\ge 1$. 		
	Consequently, by Theorem \ref{C5teo5.2}, the symbol satisfies the Gevrey-type bounds with index $\mu=\tfrac12$, and the Gabor matrix of the propagator exhibits the corresponding off-diagonal decay. 
\end{proof}

\medskip
\noindent
Note that we used the assumption \(\alpha>0\). If \(\alpha=0\) (Schr\"odinger free particle) the estimate \eqref{LRintroA10} is false. 

Theorem \ref{teor:Luigi}  refines and improves the estimate above.
\begin{cor}[Corollary of Theorem \ref{teor:Luigi}]\label{cor:teor:Luigi}
	The Gabor matrix of the propagator $\sigma(t,D_x)$ in \eqref{eq:evolheat} with $\alpha>0$ satisfies
	\begin{equation}\label{eq:heat-migliore}
		\left| \langle \sigma(t,D_x) \, \pi(z)g,\, \pi(w)g \rangle \right|
		\leq C\, e^{-\varepsilon (|z_2|^{2} + |w_2|^{2})}
		e^{-\varepsilon |z_1 -w_1|^{2}},\quad z=(z_1,z_2)\quad w=(w_1,w_2)\in\rdd.
	\end{equation} 
\end{cor}
\begin{proof}
The propagator $\sigma(t,D_x)$ has symbol $\sigma_t(\xi)=e^{-4\pi^2\gamma t|\xi|^2}$ in the Gelfand-Shilov class $S^{1/2}_{1/2}(\rd)$, so that it satisfies \eqref{eq:simbsigma} with $\mu=\nu=1/2$ and \eqref{Luigi2} gives, for $\nu=\mu=1/2$, the estimate \eqref{eq:heat-migliore}.
\end{proof}

Subsection \ref{subsec4.3} below contains the explicit computations of the Gabor matrix for the heat operator $\sigma(t,D_x)$ in \eqref{eq:evolheat}.
\subsection{Gabor matrix of the propagator}\label{subsec4.3}
In what follows we compute explicitly the Gabor matrix of the evolution operator  $\sigma(t,D_x)$ in \eqref{eq:evolheat}, with kernel $K_t$ defined in \eqref{eq:kt}, where  $\gamma=\alpha+i\beta$ as in \eqref{eq:gamma},
\(\alpha\ge 0\) and \(\beta\in\mathbb{R}\).
The Gabor matrix  of $\sigma(t,D_x)$ with respect to the Gaussian window $g(t)=e^{-\pi t^2}$, $t\in\rd$,  is given by
\[
G_t(z,w)=\big\langle {\sigma(t,D_x)}(\pi(z)g),\,\pi(w)g\big\rangle,\qquad w,z\in\mathbb{R}^{2d}.
\]
For time-frequency shifts, for \(x,\eta\in\mathbb{R}^d\), we have (see, e.g., \cite[Chapter 1]{Elena-book})  
\[
\widehat{M_\eta T_x g}(\xi)=e^{2\pi i(\eta-\xi)\cdot x}\,\widehat g(\xi-\eta).
\]
Since \(g\) is the Gaussian, \(\widehat g=g\) and we obtain
\begin{equation}\label{eq:tfs}
	\widehat{\pi(z)g}(\xi)=e^{2\pi i (z_2-\xi)\cdot z_1}\,g(\xi-z_2),\qquad
	\widehat{\pi(w)g}(\xi)=e^{2\pi i (w_2-\xi)\cdot w_1}\,g(\xi-w_2).
\end{equation}

$
\widehat{K_t}(\xi)=e^{-4\pi^2\,\gamma\,t\,|\xi|^2},
$
Parseval's identity, and formulas \eqref{eq:tfs} give
\begin{align}
	G_t(z,w)
	&=\int_{\mathbb{R}^d} \widehat{K_t}(\xi)\,
	e^{2\pi i (z_2-\xi)\cdot z_1}\,g(\xi-z_2)\,
	\overline{e^{2\pi i (w_2-\xi)\cdot w_1}\,g(\xi-w_2)}\,d\xi \nonumber\\
	&=\int_{\mathbb{R}^d} e^{-4\pi^2\gamma t|\xi|^2}\,
	e^{2\pi i (z_2-\xi)\cdot z_1}\,g(\xi-z_2)\,
	e^{-2\pi i (w_2-\xi)\cdot w_1}\,g(\xi-w_2)\,d\xi.\label{eq:int}
\end{align}

We observe that
\[
g(\xi-z_2)g(\xi-w_2)
=\exp\!\big(-\pi|\xi-z_2|^2-\pi|\xi-w_2|^2\big).
\]

Collecting the exponential terms we obtain
\begin{equation*}
	-\pi|\xi-z_2|^2-\pi|\xi-w_2|^2-4\pi^2\gamma t\,|\xi|^2 
	= -2\pi\rho_t\,|\xi|^2 + 2\pi\,\xi\cdot(z_2+w_2) - \pi(|z_2|^2+|w_2|^2),
\end{equation*}
where we set
\begin{equation}\label{eq:rho}
	\rho_t:=1+2\pi\,\gamma\,t.
\end{equation}
Moreover,
\[
\exp(2\pi i (z_2-\xi)\cdot z_1)\,\exp(-2\pi i (w_2-\xi)\cdot w_1)
=\exp(2\pi i(z_2\cdot z_1 - w_2\cdot w_1))\,\exp(\pi i\,\xi\cdot (w_1-z_1)).
\]
We therefore introduce the complex vector
\begin{equation}\label{eq:c-def}
	c:= (z_2+w_2)+i\,(w_1-z_1)\in\mathbb{C}^d.
\end{equation}
With these notations the integrand in \eqref{eq:int} can be rewritten as
\begin{equation}\label{eq:integrando}
	\exp\!\Big(-2\pi\rho_t\,|\xi|^2 + 2\pi\,c\cdot \xi\Big)\,
	\exp\!\Big(-\pi(|z_2|^2+|w_2|^2)\Big)\,
	\exp\!\Big(2\pi i(z_2\cdot z_1 - w_2\cdot w_1)\Big).
\end{equation}

Applying Lemma~\ref{lem:gauss} to \eqref{eq:integrando} with \(\rho=\rho_t\) we obtain
\begin{equation}\label{eq:Gt-forma1}
	G_t(z,w)=(2\rho_t)^{-d/2}\,
	\exp\!\Big(-\pi(|z_2|^2+|w_2|^2)\Big)\,
	\exp\!\Big(2\pi i(z_2\cdot z_1 - w_2\cdot w_1)\Big)\,
	\exp\!\Big(\frac{\pi}{2\rho_t}\, c\cdot c\Big),
\end{equation}
where $c$ is defined in \eqref{eq:c-def}.  \par
In what follows we estimate the Gabor matrix from above, computing the explicit parameter $\eps$ appearing  in \eqref{eq:heat-migliore}.
\begin{proposition}[Estimates for the Gabor matrix]\label{prop:5.3}
	For $\alpha>0$, $t>0$, $\beta\in\bR$, define
	\begin{equation}\label{eq:eps}
		\varepsilon=\varepsilon(t,\alpha,\beta)=\frac{\pi}{4} \left(1 - \sqrt{1 - \frac{8 \pi \alpha t}{(1 + 2 \pi \alpha t)^2 + (2 \pi \beta t)^2}} \right).
	\end{equation} Then, the Gabor matrix $G_t(z,w)$ can be estimated as follows:
	\begin{equation}\label{eq:GaborM}
		\left|G_t(z, w)\right| \leq \frac 1{2^{d/2}{[(1 + 2\pi \alpha t)^2 + (2\pi \beta t)^2 ]}^{d/4}} \exp\left(-\frac\varepsilon 2(|z_2|^2 + |w_2|^2) - \varepsilon |z_1 - w_1|^2\right),
	\end{equation}
	with the parameter {$\eps$} defined in \eqref{eq:eps}.
\end{proposition}

\begin{proof}
	Observe that
	$$\left|G_t(z, w)\right| = (2 |\rho_t|)^{-d/2} \exp\left( \Re\left[ \frac{\pi c \cdot c}{2\rho_t} - \pi (|z_2|^2 + |w_2|^2) \right] \right),$$
	
	where $ \rho_t$ is defined in \eqref{eq:rho}  and $c$ in \eqref{eq:c-def}.
	
	Let $a = z_2 + w_2$, $b = w_1 - z_1$, $d = z_2 - w_2$. Then the real exponent is the quadratic form
	
	$$Q = (\Re \delta - \pi/2) |a|^2 - (\Re \delta) |b|^2 - 2 (\Im \delta) (a \cdot b) - (\pi/2) |d|^2,$$
	
	where $\delta = \pi / (2\rho_t)$. 
	Observe that the term $$|z_2|^2 + |w_2|^2 = \frac{1}{2} (|a|^2 + |d|^2)$$ contributes to the decay via the coefficients of $|a|^2$ and $|d|^2$ in $Q$. From the expression for $Q$ we can infer that the coefficient of $|a|^2$ is $\Re \delta - \frac{\pi}{2}$,
	whereas the coefficient of $|d|^2$ is $-\frac{\pi}{2}$.
	
	The term $|z_1 - w_1|^2 = |b|^2$ is governed by the coefficient of $|b|^2$, which is $-\Re \delta$.
	Let us compute $\Re \delta$. We have
	$$\rho_t = 1 + 2\pi (\alpha + i\beta) t = 1 + 2\pi \alpha t + 2\pi i \beta t,$$
	so that 
	$$|\rho_t|^2 = (1 + 2\pi \alpha t)^2 + (2\pi \beta t)^2.$$
	Hence
	$$\Re \delta = \frac{\pi}{2} \cdot \frac{1 + 2\pi \alpha t}{|\rho_t|^2}.$$
	Finally, for $\alpha > 0,$ $t > 0$, \begin{equation}\label{eq:Q}
		Q \leq -\varepsilon (|a|^2 + |b|^2 + |d|^2),
	\end{equation}
	with {$\eps$} defined in \eqref{eq:eps}. From that it follows that the Gabor matrix can be controlled by the right-hand side of \eqref{eq:GaborM}.
	
\end{proof}

The value of $\varepsilon={\eps}(t,\alpha,\beta)$ in \eqref{eq:eps} satisfies $$0 < \varepsilon \leq \frac{\pi}{4},\quad  \alpha > 0,\, t > 0,\, \beta \in \mathbb{R},$$
and  the equality $$\varepsilon= \frac{\pi}{4}$$
is achieved when $$\beta = 0\quad\mbox{ and }\quad t = \frac{1}{2\pi \alpha}.$$

\noindent
As shown in Figure~\ref{fig:gabor-real-heat-t5}, the Gabor matrix of the real heat propagator
exhibits \emph{anisotropic Gaussian decay} in phase space: rapid in position shift $y$,
slower in frequency shift $\eta$:\[
|G_t((0,0),(y,\eta))|
\;\lesssim
\,\exp\!\Big(-\tfrac{\varepsilon(t)}{2}\eta^2 - \varepsilon(t)y^2\Big).
\] This asymmetry, governed by the parameter $\varepsilon(t):=\varepsilon(t,1,0)$,
reflects the intrinsic scaling of diffusion and ensures \emph{off-diagonal decay}
of the form \eqref{off-diag}, confirming the sparsity of the Gabor representation
even at large times:\par
\begin{itemize}
	\item \textbf{$t = 0$}: Sharp, isotropic Gaussian, peak $=1$.
	\item \textbf{$t = 1$}: Slight elongation along $\eta$, peak $\approx 0.50$.
	\item \textbf{$t = 2$}: Clear elliptical shape, peak $\approx 0.35$.
	\item \textbf{$t = 5$}: \emph{Strongly anisotropic}: narrow in $y$, broad in $\eta$, peak $\approx 0.18$. 
\end{itemize}
\begin{figure}[htbp]
	\centering
	\includegraphics[width=\textwidth]{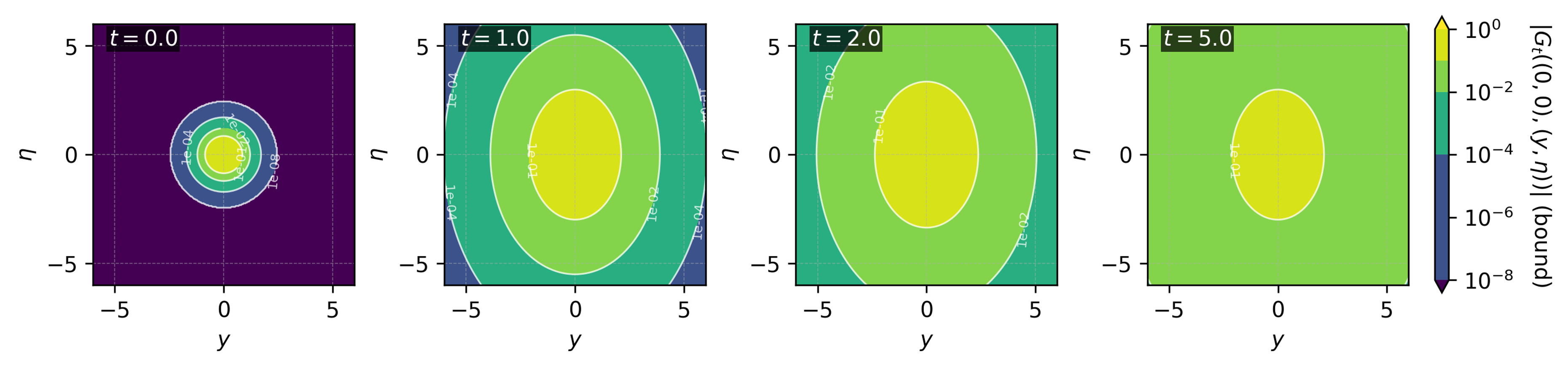}
	\caption{	Upper bound on the Gabor matrix elements $|G_t((0,0),(y,\eta))|$ 
		from \eqref{eq:GaborM} for the real heat equation ($d=1$, $\alpha=1$, $\beta=0$)
		at times $t = 0, 1, 2, 5$.
		A \emph{logarithmic color scale} \texttt{levels=[1e-8,1e-4,1e-2,1e-1,1]} and \emph{white contour lines}
		ensure visibility across the full dynamic range.
		At $t=0$, the bound matches the exact Gabor overlap (peak $=1$).
		As $t$ increases, the support spreads \emph{anisotropically}:\\
			\emph{Fast decay in $y$} (rate $\varepsilon(t,1,0)$): narrow in position shift.\\
	\emph{Slow decay in $\eta$} (rate $\varepsilon(t,1,0)/2$): broad in frequency shift.
		This anisotropic decay confirms the \emph{sparsity} of the Gabor representation 
		for parabolic propagators.}
	\label{fig:gabor-real-heat-t5}
\end{figure}

Figure \ref{fig:gabor-complex-heat} (top) considers the complex heat equation with parameters $\alpha=1$ and $\beta=1$, hence $\gamma=1+i$, in dimension $d=1$.
The picture displays the magnitude of the Gabor matrix element
\begin{equation}\label{GmatrixElements}
\bigl\lvert\langle \sigma(t,D_x)\,\pi(0,0)g,\;\pi(y,\eta)g\rangle\bigr\rvert,
\end{equation}
with Gaussian window $g(t)=e^{-\pi t^{2}}$, as in \eqref{eq:GaborM}.
Namely, we have the estimate:\begin{equation}\label{eq:GaborM-specialized}
	\left|G_t\bigl((0,0),(y,\eta)\bigr)\right|
	\;\le\;
	\frac{1}{\sqrt{2}\,\Bigl((1+2\pi t)^2+(2\pi t)^2\Bigr)^{1/4}}
	\exp\!\left(-\varepsilon(t,1,1)\,y^2 - \frac{\varepsilon(t,1,1)}{2}\,\eta^2\right),
\end{equation}
The behavior of the Gabor matrix elements \eqref{GmatrixElements} in the case where the imaginary part $\beta$ dominates over the real part $\alpha$ is displayed in Figure \ref{fig:gabor-complex-heat} (bottom) for a comparison.

\begin{figure}[t]
	\centering
	\includegraphics[width=\linewidth]{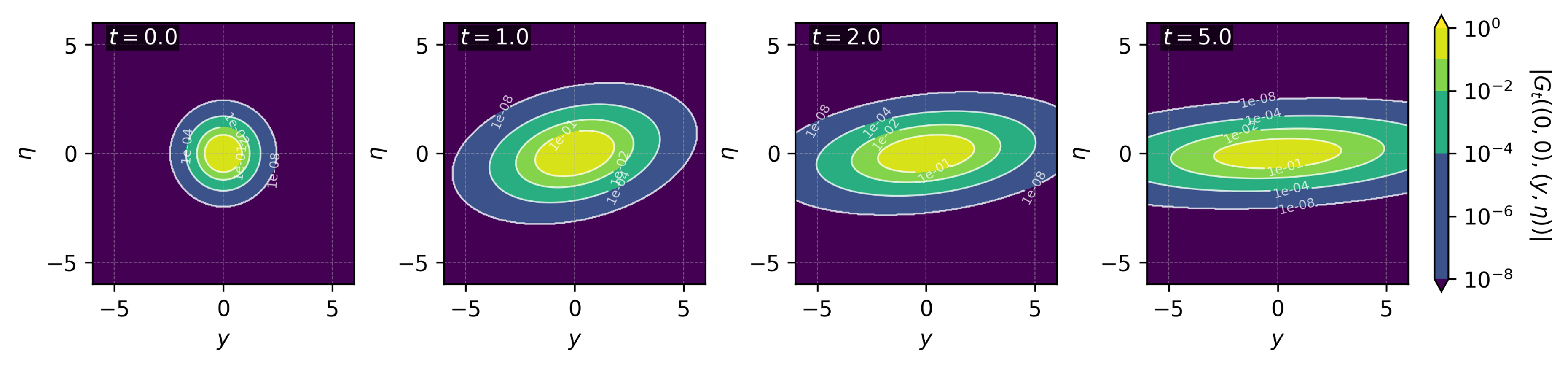}\\
	\includegraphics[width=\linewidth]{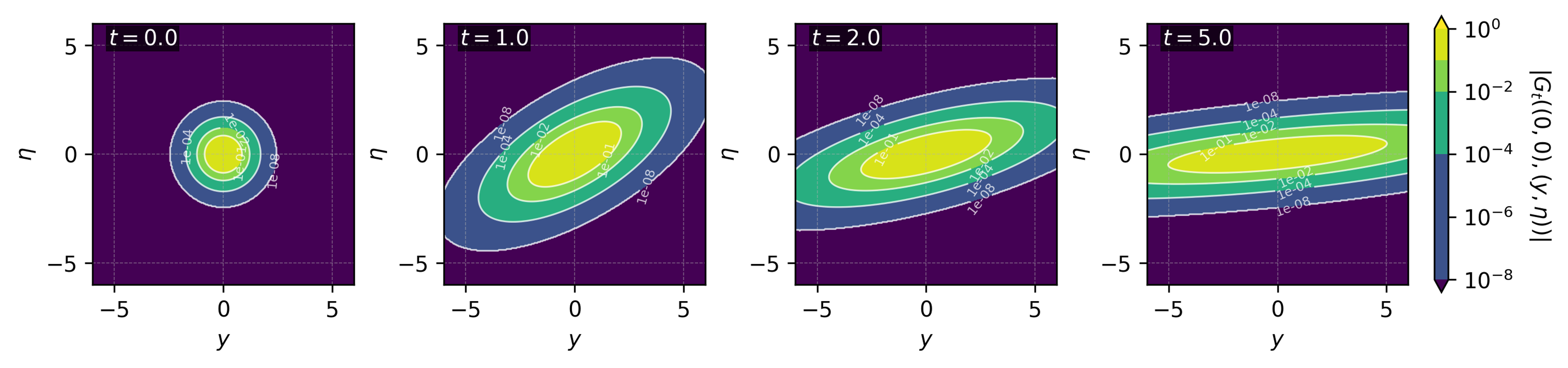}\\
	\caption{The magnitude of the Gabor matrix elements $G_t((0,0),(y,\eta))$ 
		from \eqref{eq:Gt-forma1} with parameters $d=1$, $\alpha+i\beta=1+i$ (top) and $\alpha+i\beta=0.1+i$ (bottom) at times $t = 0, 1, 2, 5$. A \emph{logarithmic color scale} \texttt{levels=[1e-8,1e-4,1e-2,1e-1,1]} and \emph{white contour lines} ensure visibility across the full dynamic range. The predominance of the {\em heat parameter} $\alpha$ with respect to the {\em Schr\"odinger parameter $\beta$} reflects on the angle between the horizontal axis and the axle shafts, which increases as $\alpha$ reduces. Moreover as $t$ increases, the support spreads \emph{anisotropically} along the axle shafts.}
	\label{fig:gabor-complex-heat}
\end{figure}

\subsection*{Comparison between Real and Complex Heat Bounds}

Figure~\ref{fig:gabor-real-heat-t5} shows the bound for the \emph{real} heat equation
($\alpha=1$, $\beta=0$), while Figure~\ref{fig:gabor-complex-heat}
illustrates the corresponding bound for the \emph{complex} heat case
($\alpha=\beta=1$ and $\alpha=0.1$, $\beta=1$), both in dimension $d=1$ and evaluated at times
$t=0,1,2,5$.

\paragraph{Analytical form.}
Both bounds derive from the same inequality
\begin{equation*}
	|G_t((0,0),(y,\eta))|
	\;\le\;
	\frac{1}{\sqrt{2}\,D(t)^{1/4}}
	\exp\!\left(-\varepsilon(t)\,y^2 - \frac{\varepsilon(t)}{2}\eta^2\right),
\end{equation*}
but with different denominators and diffusion rates:
\[
D(t)=
\begin{cases}
	(1+2\pi\alpha t)^2 & \text{(real case, }\beta=0),\\[4pt]
	(1+2\pi\alpha t)^2+(2\pi\beta t)^2 & \text{(complex case, }\beta\neq0).
\end{cases}
\]
Hence, for identical $\alpha$ and $t$, the complex heat case yields a
larger $D(t)$ and consequently a smaller prefactor and a weaker
diffusion rate $\varepsilon(t)$.

\paragraph{Quantitative comparison.}
At $(y,\eta)=(0,0)$ the peak values are ($\alpha=1$):

\begin{center}
	\begin{tabular}{c|cccc}
		$t$ & 0 & 1 & 2 & 5 \\ \hline
		Real heat ($\beta=0$) & 1.000 & 0.262 & 0.192 & 0.124 \\
		Complex heat ($\beta=1$) & 1.000 & 0.228 & 0.164 & 0.105
	\end{tabular}
\end{center}

The difference grows with $t$: the complex heat case exhibits roughly
a 15\% lower amplitude and slightly broader Gaussian profiles.

\paragraph{Geometric features.}
Both bounds produce \emph{axis-aligned} ellipses in the $(y,\eta)$-plane,
corresponding to purely anisotropic diffusion:
fast decay in $y$ (rate $\varepsilon$) and slower decay in $\eta$
(rate $\varepsilon/2$).
No rotation appears in the bound, since its exponent lacks the
mixed term $y\eta$ responsible for the phase-space tilt.
The rotation described in the complex case
(Figure~\ref{fig:gabor-complex-heat}) emerges only in the
\emph{exact} Gabor matrix, where the imaginary part of $\gamma=\alpha+i\beta$
introduces such coupling.

To sum up, the real heat equation produces a tighter, higher-amplitude ellipse,
while the complex heat equation yields a broader, lower one.
Both confirm anisotropic diffusion; only the full complex propagator
(not its bound) shows the expected symplectic rotation.

\subsection{Wigner Distribution and Wigner Kernel of the Complex Heat Equation}
In what follows we compute explicitly  the Wigner distribution of the solution. Then, we will focus on its Wigner kernel.\par
\noindent
\textbf{Wigner Distribution  of the Solution.}
To compute the Wigner distribution of $u(t,x)$, we write:
\[
Wu(t,x,\xi) = \int_{\mathbb{R}^d} u\left(t, x + \frac{y}{2}\right) \overline{u\left(t, x - \frac{y}{2}\right)} e^{-2\pi i \xi \cdot y} \, dy.
\]

Using the linearity and convolution properties of the Wigner transform, for $\alpha\not=0$ and $t>0$:
\begin{align}\label{e11}
	{W}(u(t, \cdot))(x,\xi) &= W\left( \mathscr{F}^{-1} \left( e^{-4\pi^2 (\alpha + i\beta)t |\cdot|^2} \hat{u}_0 \right) \right)(x,\xi)\\
	&=W\left( \mathscr{F}^{-1}  \left[e^{-4\pi^2 i\beta t |\cdot|^2} \mathscr{F}\mathscr{F}^{-1}\left( e^{-4\pi^2 \alpha t |\cdot|^2} \hat{u}_0 \right) \right]\right)(x,\xi)\\
	&=\frac{1}{(4\pi \alpha t)^{d/2}}W\left( \mathscr{F}^{-1} e^{-4\pi^2 i\beta t |\cdot|^2} \mathscr{F}(e^{-\frac{|\cdot|^2}{4\alpha t}}\ast {u}_0) \right)(x,\xi) \\
	&=\frac{1}{(4\pi \alpha t)^{d/2}}W\left( e^{-\frac{|\cdot|^2}{4\alpha t}}\ast {u}_0\right)(x-4\pi\beta t\xi,\xi).
\end{align}

We now compute the Wigner distribution of the Gaussian-smoothed function:
\[
W\left( \varphi_t \ast u_0 \right)(x, \xi),
\quad \text{where} \quad \varphi_t(x) = e^{-\frac{|x|^2}{4\alpha t}}.
\]
If \( v = \varphi \ast u_0 \), then the Wigner distribution satisfies:
\[
Wv(x,\xi) = W\varphi(x,\xi) \ast_x Wu_0(x,\xi),
\]
that is, convolution in the position variable \( x \).
The Wigner of the Gaussian $\varphi_t$ is given by (see, e.g., \cite[Lemma 1.3.12]{Elena-book})
\[
W{\varphi_t}(x, \xi) = (8\pi \alpha t)^{d/2} \, e^{-\frac{|x|^2}{2\alpha t}} e^{-8\pi^2 \alpha t |\xi|^2}.
\]
Hence, the Wigner distribution of the smoothed function is:
\[
W\left( \varphi_t \ast u_0 \right)(x,\xi) =  W\varphi_t(x,\xi) \ast_x Wu_0(x,\xi),
\]
that is:
\[
W\left( \varphi_t \ast u_0 \right)(x,\xi) 
= (8\pi \alpha t)^{d/2} \left( \int_{\mathbb{R}^d} Wu_0(x - y, \xi) \, e^{-\frac{|y|^2}{2\alpha t}} \, dy \right) 
\cdot e^{-8\pi^2 \alpha t |\xi|^2}.
\]
Finally, 
\begin{equation}\label{WignerOpHeat}
	{W}(u(t, \cdot))(x,\xi) = 2^{d/2} \left( \int_{\mathbb{R}^d} Wu_0(x-4\pi\beta t \xi - y, \xi) \, e^{-\frac{|y|^2}{2\alpha t}} \, dy \right) 
	\cdot e^{-8\pi^2 \alpha t |\xi|^2}.
\end{equation}
Observe that the result encapsulates both dissipative and oscillatory behavior governed by $\alpha$ and $\beta$, respectively.

\textbf{Wigner Kernel of the Complex Heat Equation.} The theory developed above can be applied to the solution operator of the complex heat equation, having fundamental solution
$$E(t,\cdot)=\exp\left( -\frac{|\cdot|^2}{4(\alpha + i\beta)t} \right),\quad \alpha\geq0, \,t>0.$$

First, we consider the case $\alpha>0$. Writing  \(\gamma=\alpha+i\beta\) as in \eqref{eq:gamma},
\[
E\!\left(t,x+\tfrac{u}{2}\right)\overline{E\!\left(t,x-\tfrac{u}{2}\right)}
=\exp\!\left(
-\frac{1}{4t}\Big[(|x|^2+\tfrac{|u|^2}{4})\Big(\tfrac{1}{\gamma}+\tfrac{1}{\bar \gamma}\Big)
+(x\cdot u)\Big(\tfrac{1}{\gamma}-\tfrac{1}{\bar \gamma}\Big)\Big]\right).
\]

Since
\[
\frac{1}{\gamma}+\frac{1}{\bar \gamma}=\frac{2\alpha}{\alpha^2+\beta^2},
\qquad
\frac{1}{\gamma}-\frac{1}{\bar \gamma}=-\frac{2i\beta}{\alpha^2+\beta^2},
\]
we obtain
\[
E\!\left(t,x+\tfrac{u}{2}\right)\overline{E\!\left(t,x-\tfrac{u}{2}\right)}
=\exp\!\left(-\frac{\alpha}{2t(\alpha^2+\beta^2)}|x|^2\right)\,
\exp\!\left(-\frac{\alpha}{8t(\alpha^2+\beta^2)}|u|^2\right)\,
\exp\!\left(i\,\frac{\beta}{2t(\alpha^2+\beta^2)}\,x\cdot u\right).
\]

Thus
\[
W({E(t,\cdot)})(x,\xi)
=\exp\!\left(-\frac{\alpha}{2t(\alpha^2+\beta^2)}|x|^2\right)
\int_{\mathbb{R}^d}
\exp\!\left(-a|u|^2+i\,b\cdot u\right)\,du,
\]
with
\[
a=\frac{\alpha}{8t(\alpha^2+\beta^2)}, 
\qquad
b=\frac{\beta}{2t(\alpha^2+\beta^2)}\,x-2\pi \xi.
\]

Using the Gaussian integral
\[
\int_{\mathbb{R}^d} e^{-a|u|^2+i\,b\cdot u}\,du
=\left(\frac{\pi}{a}\right)^{d/2}
\exp\!\left(-\frac{|b|^2}{4a}\right), \qquad a>0,
\]
we obtain
\[
	W({E(t,\cdot)})(x,\xi)
	=\left(\frac{8\pi t(\alpha^2+\beta^2)}{\alpha}\right)^{\!d/2}
	\exp\!\left(
	-\frac{\alpha}{2t(\alpha^2+\beta^2)}|x|^2
	-\frac{2t(\alpha^2+\beta^2)}{\alpha}\,
	\Big|\,2\pi \xi-\tfrac{\beta}{2t(\alpha^2+\beta^2)}\,x\Big|^{2}
	\right).
\]


Set
\[
a:=\frac{\alpha}{2t(\alpha^2+\beta^2)},\qquad
\kappa:=\frac{\beta}{4\pi t(\alpha^2+\beta^2)},\qquad
b:=\frac{8\pi^2\,t(\alpha^2+\beta^2)}{\alpha}.
\]
Then
\begin{equation}\label{eq:Wigner-fond-sol}	W({E(t,\cdot)})(x,\xi)
	=\Big(\tfrac{8\pi t(\alpha^2+\beta^2)}{\alpha}\Big)^{\!d/2}
	\exp\!\big(-a\,|x|^2\big)\;
	\exp\!\big(-b\,|\xi-\kappa x|^2\big).
\end{equation}
Equivalently, \(W({E(t,\cdot)})\) is a Gaussian in \((x,\xi)\) centered on the affine subspace
\(\xi=\kappa x\), with quadratic form diagonal in the variables \(\{x,\ \xi-\kappa x\}\).

\medskip
For the Wigner kernel we obtain
\begin{align}	&k_W(x,\xi,y,\eta)=\notag\delta(\xi-\eta)\,
	\Big(\tfrac{8\pi t(\alpha^2+\beta^2)}{\alpha}\Big)^{\!d/2}
	\exp\!\big(-a\,|x-y|^2\big)\;
	\exp\!\big(-b\,|\xi-\kappa (x-y)|^2\big)\notag\\
	&=\delta(\xi-\eta)\,
	\Big(\tfrac{8\pi t(\alpha^2+\beta^2)}{\alpha}\Big)^{\!d/2}
	\exp\!\Big(-\frac{\alpha}{2t(\alpha^2+\beta^2)}\,|x-y|^2\Big)\;
	\exp\!\Big(-\frac{8\pi^2\,t(\alpha^2+\beta^2)}{\alpha}\,|\xi-\kappa (x-y)|^2\Big).\label{eq:Wigkernheat}
\end{align}

If \(\alpha=0\) and $\beta\not=0$ we are in the case of the Schr\"odinger free particle. Here the solution can be expressed as the action of a metaplecric operator $\widehat{S_t}\in \Mp(d,\bR)$ as follows: $$u(t, \cdot)=\widehat{S_t}u_0,$$ where the related symplectic matrix $S_t\in \mbox{Sp}(d,\bR)$ is given by
\begin{equation}\label{decompS}
	S_t=\begin{pmatrix}
		I & 4\pi \beta tI\\ O & I
	\end{pmatrix},
\end{equation}
see, e.g., \cite{Knutsen}. Using Lemma \ref{lem:cov},  the Wigner of $u(t,\cdot)=E(t,\cdot)\ast u_0$ is given by 
$$W(u(t,\cdot))(x,\xi)=Wu_0(S^{-1}\phas)=Wu_0(x-4\pi\beta t \xi,\xi),$$
according to \eqref{e11} for $\alpha=0$.
The Wigner kernel $k_W$ was already computed in \cite[Section 7]{CRGFIO1}:
\begin{equation*}
	k_W(x,\xi,y,\eta)=\delta\left({(x,\xi)-S_t(y,\eta)}\right)=\delta(x-y+4\pi \beta \eta ,\xi -\eta).
\end{equation*}
\subsubsection{Real heat kernel (\(\beta=0\))}
Let \(E(t,x)=\exp\!\big(-|x|^2/(4\alpha t)\big)\) with \(\alpha>0,\ t>0\).
Specializing the general formula \eqref{eq:Wigner-fond-sol} to \(\beta=0\) yields
\[	W({E(t,\cdot)})(x,\xi)
=(8\pi\alpha t)^{d/2}\,
\exp\!\left(
-\frac{|x|^2}{2\alpha t}\;-\;8\pi^{2}\alpha t\,|\xi|^{2}
\right),
\;
\]
so that the Wigner kernel simplifies to
\[
	k_W(x,\xi,y,\eta)
	=
	\delta(\xi-\eta)\,
	(8\pi\alpha t)^{d/2}\,
	\exp\!\left(
	-\frac{|x-y|^2}{2\alpha t}
	-8\pi^2 \alpha t\,|\xi|^2
	\right).
\]
\section{The Wave Equation}\label{sec:waveEq}

Let us now consider the Cauchy problem for the wave equation:
\begin{equation} \label{eq:wave}
	\partial_{tt} u - \Delta_x u = 0, \quad (t, x) \in \mathbb{R} \times \mathbb{R}^d,
\end{equation}
with initial data
\[
u(0, x) = u_0(x), \qquad \partial_t u(0, x) = u_1(x).
\]
The solution can be expressed in the form
\begin{equation} \label{eq:solution}
	u(t, x) = T_t u_1(x) + \partial_t T_t u_0(x),
\end{equation}
where $T_t$ is the Fourier multiplier
\begin{equation} \label{eq:Tt}
	T_t f(x) = \int_{\mathbb{R}^d} e^{2\pi i x \cdot \xi} \, \sigma(t,\xi) \, \widehat{f}(\xi) \, d\xi,
\end{equation}
with symbol
\begin{equation*} 
	\sigma(t,\xi) = \frac{\sin(2\pi |\xi| t)}{2\pi |\xi|}, \qquad \xi \in \mathbb{R}^d\setminus\{0\}.
\end{equation*}
The symbol $\sigma(t,\xi)$ satisfies the condition \eqref{eq:simbsigma}
 with $\mu=0$ and $\nu=\infty$. In this case, both Theorem \ref{teor:Luigi} and \ref{C5teo5.2}  give the same Gabor matrix decay:
$$|\la \sigma(t,D_x) \pi(z)g,\pi(w)g\ra|\leq C e^{{-\eps |w-z|}^{2}},\quad z,w\in\rdd,$$
where $g$ is the Gaussian window (or any window $g\in S^{1/2}_{1/2}$). 

In the following, limiting our attention to dimension $d\leq 3$, we shall first discuss the Wigner kernel $k_W$. A preliminary result is given by Theorem \ref{LRthm2circ}, showing that $k_W$ is compactly supported in the $s=x-y$ variable, see \eqref{LReqK9} and the comments at the end of Section \ref{sec:DecayingSymbols}.

\subsection{Wigner kernel of the wave equation for $d\leq 3$}

For the specific wave multiplier $\sigma(t,D_x)$, $d\leq 3$,
we have additional information that allows us to compute the Wigner kernel explicitly.  Let $d\le 3$, then
\begin{equation}\label{eq:sigma}
(\sigma(t,D_x)f)(x)=\int_{\mathbb{R}^d} f(x-y)\,d\mu_t(y)=\int_{\mathbb{R}^d} f(x-y)\,m_t(y)dy,
\end{equation}
where \begin{equation}\label{eq:mt}
\mu_t(y)=m_t(y)dy
\end{equation}
 is a positive Borel measure supported in $B_t(0)$ with total mass $t$ (see, e.g., \cite[Chapter 4]{Rauch}), and we denoted by $m_t(y)$ its density with respect to Lebesgue measure $dy$. 

In detail,  for $d=1$, $d\mu_t(y)=\tfrac12\chi_{[-t,t]}(y)\,dy$;
for $d=2$, $d\mu_t(y)=(2\pi)^{-1}(t^2-|y|^2)_+^{-1/2}\,dy$;
for $d=3$, $d\mu_t(y)=(4\pi t)^{-1}\,d\sigma_{\partial B_t(0)}$.
\begin{proposition}\label{prop:wigner-wave-correct}
	Let $d\le 3$ and $\sigma(t,D_x)f)(x)$ in \eqref{eq:sigma}.
	Then for all $f,h\in\mathcal{S}(\mathbb{R}^d)$ one has
	\begin{equation}\label{eq:wigner-conv-correct}
		W(\sigma(t,D_x)f,\sigma(t,D_x)h)(x,\xi)
		=\iint_{\mathbb{R}^d\times\mathbb{R}^d} e^{-2\pi i\,(y-y')\cdot\xi}\,
		W(f,h)\!\left(x-\tfrac{y+y'}{2},\,\xi\right)\, d\mu_t(y)\,d\mu_t(y').
	\end{equation}
\end{proposition}

\begin{proof}
Set $T_t=\sigma(t,D_x)$.	Since $\mu_t$ is positive we may apply Fubini–Tonelli:
	\[
	\begin{aligned}
		W(T_t f,T_t h)(x,\xi)
		&=\int_{\mathbb{R}^d} e^{-2\pi i\,u\cdot\xi}\,
		(T_t f)(x+\tfrac{u}{2})\,\overline{(T_t h)(x-\tfrac{u}{2})}\,du\\
		&=\iint_{\mathbb{R}^d\times\mathbb{R}^d} d\mu_t(y)\,d\mu_t(y')
		\int_{\mathbb{R}^d} e^{-2\pi i\,u\cdot\xi}\,
		f\!\left(x-y+\tfrac{u}{2}\right)\,\overline{h\!\left(x-y'-\tfrac{u}{2}\right)}\,du.
	\end{aligned}
	\]
	Changing variables  $u=2v$, $du=2^d\,dv$:
	\[
	\int e^{-2\pi i\,u\cdot\xi}\,f\!\left(x-y+\tfrac{u}{2}\right)\,\overline{h\!\left(x-y'-\tfrac{u}{2}\right)}\,du
	=2^d\int e^{-4\pi i\,v\cdot\xi}\,f(x-y+v)\,\overline{h(x-y'-v)}\,dv.
	\]
	We recenter the arguments at $x-\tfrac{y+y'}{2}$:
	\[
	f(x-y+v)\,\overline{h(x-y'-v)}
	=f\!\left(\Bigl(x-\tfrac{y+y'}{2}\Bigr)+\tfrac{2v+y'-y}{2}\right)\,
	\overline{h\!\left(\Bigl(x-\tfrac{y+y'}{2}\Bigr)-\tfrac{2v+y'-y}{2}\right)}.
	\]
	We then set $w=2v+y'-y$; then $v=\tfrac{w-(y'-y)}{2}$ and $dv=\tfrac{dw}{2^d}$.
	The exponential factor becomes
	\[
	e^{-4\pi i\,v\cdot\xi}
	= e^{-2\pi i\,(w-(y'-y))\cdot\xi}
	= e^{-2\pi i\,w\cdot\xi}\,e^{+2\pi i\,(y'-y)\cdot\xi}.
	\]
	By substituting, we obtain
	\[
	\begin{aligned}
		&2^d\int_{\mathbb{R}^d} e^{-4\pi i\,v\cdot\xi}\,f(x-y+v)\,\overline{h(x-y'-v)}\,dv\\
		&\qquad= e^{-2\pi i\,(y-y')\cdot\xi}\int_{\mathbb{R}^d} e^{-2\pi i\,w\cdot\xi}\,
		f\!\left(\Bigl(x-\tfrac{y+y'}{2}\Bigr)+\tfrac{w}{2}\right)\,
		\overline{h\!\left(\Bigl(x-\tfrac{y+y'}{2}\Bigr)-\tfrac{w}{2}\right)}\,dw.
	\end{aligned}
	\]
	The integral in $w$ is exactly $W(f,h)\!\left(x-\tfrac{y+y'}{2},\xi\right)$. Therefore
	\[
	\int_{\mathbb{R}^d} e^{-2\pi i\,u\cdot\xi}\,
	f\!\left(x-y+\tfrac{u}{2}\right)\,\overline{h\!\left(x-y'-\tfrac{u}{2}\right)}\,du
	= e^{-2\pi i\,(y-y')\cdot\xi}\,W(f,h)\!\left(x-\tfrac{y+y'}{2},\xi\right).
	\]
Plugging this into the double average with respect to $d\mu_t(y)\,d\mu_t(y')$ we obtain
 \eqref{eq:wigner-conv-correct}. 
 
\end{proof}

Using the definition \eqref{I4},  we can now infer the Wigner kernel of $\sigma(t,D_x)$.
\begin{theorem}\label{thr:wave-correct}
The Wigner kernel of $\sigma(t,D_x)$ for $d\leq 3$, defined in \eqref{eq:sigma}, is given by
\begin{align}
k_W(x,\xi,y,\eta)
&=\delta(\eta-\xi)\int_{\mathbb{R}^d} e^{-2\pi i\,r\cdot\xi}\,
m_t\!\left(x-y+\tfrac{r}{2}\right)\,m_t\!\left(x-y-\tfrac{r}{2}\right)\,dr\\
&=\delta(\eta-\xi)Wm_t(x-y,\xi).
\end{align}
\end{theorem}
\begin{proof}
From Proposition~\ref{prop:wigner-wave-correct} we know that
\[
W(T_t f,T_t h)(x,\xi)
=\iint e^{-2\pi i\,(a-b)\cdot\xi}\,
W(f,h)\!\left(x-\tfrac{a+b}{2},\xi\right)\,d\mu_t(a)\,d\mu_t(b).
\]
We want to rewrite this as
\[
W(T_t f,T_t h)
(x,\xi)
=\iint k_W(x,\xi,y,\eta)\,W(f,h)(y,\eta)\,dy\,d\eta.
\]
Since the right-hand side of the first formula leaves the frequency variable unchanged, we must have a Dirac factor $\delta(\eta-\xi)$ in $k_W$. Moreover, the spatial argument is shifted by $\tfrac{a+b}{2}$, which we encode by a delta constraint $y=x-\tfrac{a+b}{2}$. Therefore
\[
\begin{aligned}
	k_W(x,\xi,y,\eta)
	&=\delta(\eta-\xi)\!\!\iint e^{-2\pi i\,(a-b)\cdot\xi}\,
	\delta\!\left(y-x+\tfrac{a+b}{2}\right)\,d\mu_t(a)\,d\mu_t(b).
\end{aligned}
\]
Introduce the change of variables $s=\tfrac{a+b}{2}$ and $r=a-b$ (so $a=s+\tfrac{r}{2}$, $b=s-\tfrac{r}{2}$). Then $\delta\!\left(y-x+\tfrac{a+b}{2}\right)=\delta\!\left(y-x+s\right)$, and
\[
k_W(x,\xi,y,\eta)
=\delta(\eta-\xi)\int e^{-2\pi i\,r\cdot\xi}\,
d\mu_t\!\left(x-y+\tfrac{r}{2}\right)\,
d\mu_t\!\left(x-y-\tfrac{r}{2}\right).
\]
Equivalently, if $\mu_t$ has a density $m_t$ with respect to Lebesgue measure,
\[
k_W(x,\xi,y,\eta)
=\delta(\eta-\xi)\int_{\mathbb{R}^d} e^{-2\pi i\,r\cdot\xi}\,
m_t\!\left(x-y+\tfrac{r}{2}\right)\,m_t\!\left(x-y-\tfrac{r}{2}\right)\,dr.
\]
\end{proof}

Define, for each $\xi\in\mathbb{R}^d$, the kernel on $s\in\mathbb{R}^d$
\[
\kappa_t(s,\xi):=\int_{\mathbb{R}^d} e^{-2\pi i\,r\cdot\xi}\,
d\mu_t\!\left(s+\tfrac{r}{2}\right)\,d\mu_t\!\left(s-\tfrac{r}{2}\right).
\]
Then
\[
k_W(x,\xi,y,\eta)=\delta(\eta-\xi)\,\kappa_t(x-y,\xi).
\]

\begin{example} If $d=1$,  then $d\mu_t(y)=\tfrac12\chi_{[-t,t]}(y)\,dy$, and, for $s\in\mathbb{R}$,
\begin{equation}\label{eq:kappa1}
\kappa_t^{(1)}(s,\xi)
=\frac{1}{4}\int_{\mathbb{R}} e^{-2\pi i r\xi}\,
\chi_{[-t,t]}\!\left(s+\tfrac{r}{2}\right)\,
\chi_{[-t,t]}\!\left(s-\tfrac{r}{2}\right)\,dr
=\chi_{\{|s|\le t\}}\,
\frac{\sin\!\big(4\pi (t-|s|)\,\xi\big)}{4\pi\,\xi}.
\end{equation}
Hence
\[
k_W(x,\xi,y,\eta)
=\delta(\eta-\xi)\,\chi_{\{|x-y|\le t\}}\,
\frac{\sin\!\big(4\pi (t-|x-y|)\,\xi\big)}{4\pi\,\xi}.
\]
See Figure \ref{fig:kappa123} for a visual representation.
\end{example}

\begin{example}\label{ex:wigner-wave-d23}
	Let $s:=x-y$, let $\widehat s:=s/|s|$ for $s\neq0$, decompose $\xi=\xi_\parallel+\xi_\perp$ with
	$\xi_\parallel:=(\xi\cdot\widehat s)\widehat s$ and $\xi_\perp:=\xi-\xi_\parallel$, and set
	\(
	r(s):=\sqrt{\,t^2-\tfrac{|s|^2}{4}\,}\,.
	\)
	Then: \\
{\emph{	Case $d=2$}.}
\begin{equation}\label{eq:kappa2}
		k_W\big(x,\xi,y,\eta\big)
		=\delta(\eta-\xi)\,\kappa_t^{(2)}(\xi;s),\,\,\,\mbox{where}\,\,\,\,
	\kappa_t^{(2)}(\xi;s)
		=\frac{\mathbf{1}_{\{|s|<2t\}}}{2\pi}\,
		J_0\!\big(4\pi\,r(s)\,|\xi_\perp|\big).
\end{equation}
	{\emph{	Case $d=3$}.}
\begin{equation}\label{eq:kappa3-corrected}
	k_W\big(x,\xi,y,\eta\big)
	=\delta(\eta-\xi)\,\kappa_t^{(3)}(\xi;s),\quad
	\kappa_t^{(3)}(\xi;s)
	=\frac{r(s)}{8\pi\,t^2}\,\mathbf{1}_{\{|s|<2t\}}\,
	J_0\!\big(4\pi\,r(s)\,|\xi_\perp|\big).
\end{equation}
\end{example}

\begin{proof}
	\emph{Case $d=2$.}
	Here $d\mu_t(y)=\tfrac{1}{2\pi}(t^2-|y|^2)_{+}^{-1/2}\,dy$, hence
	\[
	\kappa_t^{(2)}(\xi;s)=\frac{1}{(2\pi)^2}
	\int_{\mathbb{R}^2} e^{-2\pi i r\cdot \xi}\,
	\frac{\mathbf{1}_{\{|s+r/2|<t\}}\;\mathbf{1}_{\{|s-r/2|<t\}}}
	{\sqrt{t^2-|s+r/2|^2}\,\sqrt{t^2-|s-r/2|^2}}\,dr.
	\]
	The constraints describe the intersection of the two discs of radius $t$ centered at $\pm s$, a symmetric ``lens.''  
	Placing coordinates so that $s=(|s|,0)$ and writing $r=(u,v)$, the integration in $u$ can be carried out, leaving
	\[
	\int_{-r(s)}^{r(s)} \frac{e^{-4\pi i v (\xi\cdot e_\perp)}}{\sqrt{r(s)^2-v^2}}\,dv,
	\qquad r(s)=\sqrt{t^2-\tfrac{|s|^2}{4}}.
	\]
	Using the classical identity
	\[
	\int_{-a}^{a} \frac{e^{-i kv}}{\sqrt{a^2-v^2}}\,dv=\pi J_0(ak),
	\]
	(where $J_0$ is the Bessel function of the first kind of order zero, defined in \eqref{eq:Bessel}), we obtain
	\[
	\kappa_t^{(2)}(\xi;s)=\frac{\mathbf{1}_{\{|s|<2t\}}}{2\pi}\,
	J_0\!\big(4\pi r(s)|\xi_\perp|\big).
	\]
	
	\smallskip
\emph{Case $d=3$.}
Here $d\mu_t(y)=\tfrac{1}{4\pi t}\,d\sigma_{\partial B_t(0)}(y)$.
The two delta conditions $|s\pm r/2|=t$ force $r=2v$ with $v$ in the circle
\[
\mathcal{C}(s)=\{v\in\mathbb{R}^3:\ v\perp s,\ |v|=r(s)\},\qquad |s|<2t.
\]
Therefore
\[
\kappa_t^{(3)}(\xi;s)=\frac{1}{(4\pi t)^2}\int_{\mathcal{C}(s)}
e^{-4\pi i v\cdot \xi}\,d\sigma_{\mathcal{C}}(v).
\]
The integral is the Fourier transform of the uniform measure on a circle of radius $r(s)$,
\[
\frac{1}{2\pi r(s)}\int_{\mathcal{C}(s)} e^{-4\pi i v\cdot\xi}\,d\sigma_{\mathcal{C}}(v)
=J_0\!\big(4\pi r(s)|\xi_\perp|\big).
\]
Multiplying by $2\pi r(s)$ gives
\[
\int_{\mathcal{C}(s)} e^{-4\pi i v\cdot\xi}\,d\sigma_{\mathcal{C}}(v)
=2\pi r(s)\,J_0\!\big(4\pi r(s)|\xi_\perp|\big).
\]
Thus
\[
\kappa_t^{(3)}(\xi;s)
=\frac{r(s)}{8\pi\,t^2}\,\mathbf{1}_{\{|s|<2t\}}\,
J_0\!\big(4\pi r(s)|\xi_\perp|\big).
\]
\end{proof}


\begin{figure}[ht]
	\centering
	\includegraphics[width=0.32\textwidth]{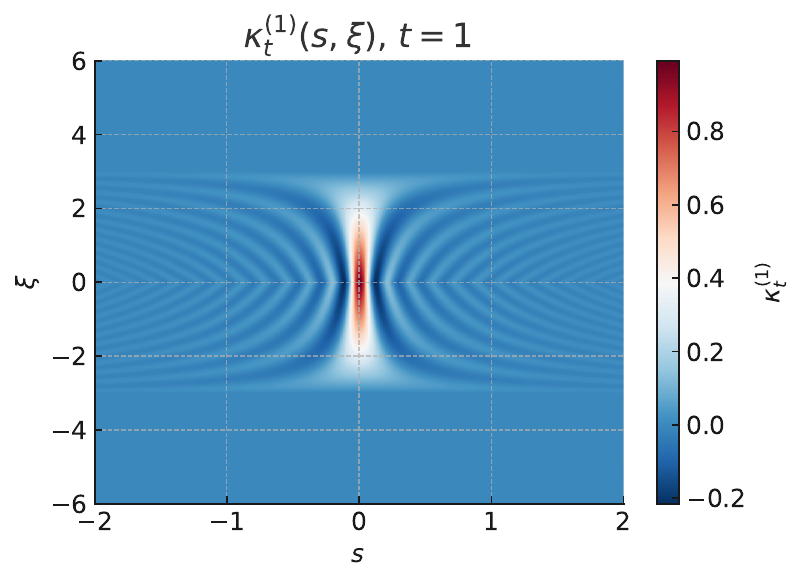}\hfill
	\includegraphics[width=0.32\textwidth]{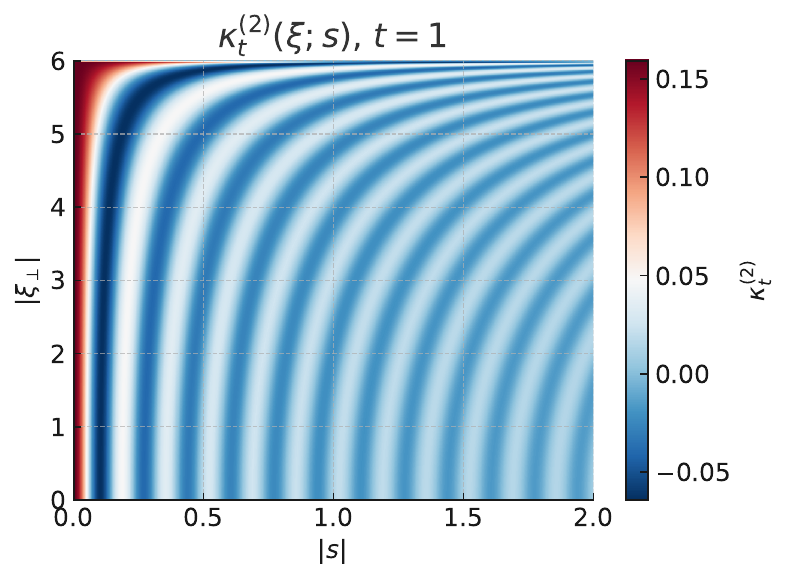}\hfill
	\includegraphics[width=0.32\textwidth]{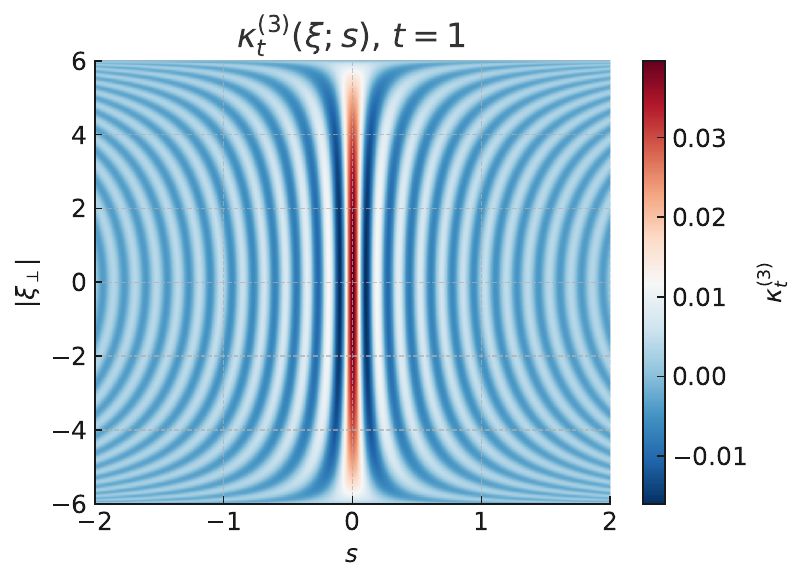}
	\caption{Kernels of Examples~4.5 and 4.6.
		Left: $\kappa_t^{(1)}(s,\xi)$ from Example~4.5.
		Center: $\kappa_t^{(2)}(\xi;s)$ in dimension $d=2$.
		Right: $\kappa_t^{(3)}(\xi;s)$ in dimension $d=3$. 
		All plots are shown for $t=1$.}
	\label{fig:kappa123}
\end{figure}

\subsection{Gabor matrix of the wave equation for $d\leq3$}

For the wave  multiplier $T_t$ in \eqref{eq:Tt}, the connection between 
the Gabor matrix and the Wigner kernel in \eqref{convgg1}, written for the Gaussian function $g(t)=e^{-\pi |t|^2}$, $t\in\rd$, gives the explicit expression of the Gabor matrix below. We refer to Figure \ref{fig:gabor-wave-one-spot} for a visual representation in dimension $d=1$.

\begin{theorem}
	The related Gabor matrix for $d \leq 3$ is given by
\begin{align}\label{eq:GM-wave}
	|\langle T_t \pi(z) g, \pi(w) g \rangle|^2 &=  2^{-d} e^{-\pi |\xi - \eta|^2} \, \int_{\mathbb{R}^d} \exp\left( -\pi |x - y - u|^2 \right) \\
	&\qquad\times\,\left[ \int_{\mathbb{R}^d} \exp\left( -\frac{\pi}{4} |r|^2 \right) m_t\left( u + \frac{r}{2} \right) m_t\left( u - \frac{r}{2} \right) e^{-2\pi i \, r \cdot \frac{\xi + \eta}{2}} \, dr \right] du,
\end{align}
	where $ z = (x, \xi) $, $ w = (y, \eta) \in \mathbb{R}^{2d} $.
\end{theorem}

\begin{proof}
	Taking the Gaussian function $\f(t)=e^{-\pi t^2}$ its  Wigner distribution is \cite[Lemma 1.3.12]{Elena-book} $W\f(x,\xi)=2^{d/2}e^{-2\pi (x^2+\xi^2)}$ so that the function $\Phi=Wg\otimes Wg$ in \eqref{convgg1} is given by 
	$$\Phi(w,z)=W\f(w)W\f(z)=2^{d} e^{-2\pi(w^2+z^2)}, \quad w,z\in\rdd.$$
	By the given relation, for $z=(x,\xi)$, $w=(y,\eta)$,
	\[
	|\langle T_t \pi(z) g, \pi(w) g \rangle|^2 = k_W \ast \Phi(z, w) = \int_{\mathbb{R}^{4d}} k_W(x',\xi',y',\eta') \, \Phi(x-x',\xi-\xi', y-y',\eta-\eta') \, dx' \, d\xi' \, dy' \, d\eta'.
	\]
	Substituting  $\Phi$ and the expression for $k_W$:
	\[
	k_W(x',\xi',y',\eta') = \delta(\eta' - \xi') \int_{\mathbb{R}^d} e^{-2\pi i r \cdot \xi'} m_t\left(x' - y' + \frac{r}{2}\right) m_t\left(x' - y' - \frac{r}{2}\right) dr.
	\]
	The integral on the right-hand side becomes
	\begin{align*}
	2^d \int \, \delta(\eta' - \xi') \, & e^{-2\pi i r \cdot \xi'} m_t\left(x' - y' + \frac{r}{2}\right) m_t\left(x' - y' - \frac{r}{2}\right)\\ 
	&\qquad\times\,  e^{-2\pi |(x-x',\xi-\xi')|^2} e^{-2\pi |(y-y',\eta-\eta')|^2} dx' \, dy' \, d\xi' \, d\eta' \, dr .
	\end{align*}
	Integrate over $\eta'$ using the delta function ($\eta' = \xi'$):
	\begin{align*}
	2^d \int e^{-2\pi i r \cdot \xi'} & m_t\left(x' - y' + \frac{r}{2}\right) m_t\left(x' - y' - \frac{r}{2}\right)\\ 
	&\qquad \times\,e^{-2\pi |x-x'|^2} e^{-2\pi |\xi-\xi'|^2} e^{-2\pi |y-y'|^2} e^{-2\pi |\eta-\xi'|^2} dx' \, dy' \, d\xi' \, dr \,.
	\end{align*}
	The integral over $\xi'$ is
\begin{align*}
	\int & \exp\left(-2\pi |\xi - \xi'|^2 - 2\pi |\eta - \xi'|^2 - 2\pi i r \cdot \xi'\right)\,d\xi' \\
	& = \int  \exp\left(-4\pi |\xi'|^2 + 4\pi (\xi + \eta) \cdot \xi' - 2\pi i r \cdot \xi' - 2\pi (|\xi|^2 + |\eta|^2)\right)\,d\xi' \\
	&=	 \int \exp\left(-4\pi |\xi'|^2 + [4\pi (\xi + \eta) - 2\pi i r] \cdot \xi' - 2\pi (|\xi|^2 + |\eta|^2)\right)\,d\xi' \\
	&=\exp(-\pi |\xi - \eta|^2) \exp\left( -\frac{\pi}{4} |r|^2 - 2\pi i r \cdot \frac{\xi + \eta}{2} \right).
\end{align*}
The last equality is obtained completing the square (with complex linear term). Namely, let $$p = 4\pi (\xi + \eta), \quad q = -2\pi r, \quad a = 4\pi,$$
 then  $$2^{-d} \exp\left( \frac{|p|^2 - |q|^2}{4a} + i \frac{p \cdot q}{2a} \right) = 2^{-d} \exp\left( \pi |\xi + \eta|^2 - \frac{\pi}{4} |r|^2 - 2\pi i r \cdot \frac{\xi + \eta}{2} \right).$$
	The remaining integral over $x', y'$ is computed as follows.  Set $u = x' - y'$. Then it becomes
	\[
	\int \, \left[ \int  \, \exp\left( -2\pi |x - x'|^2 - 2\pi |y - (x' - u)|^2 \right) \,dx'\right] m_t\left(u + \frac{r}{2}\right) m_t\left(u - \frac{r}{2}\right) du.
	\]
	The inner integral over $x'$ is \begin{align*}
		\int & \, \exp\left( -4\pi |x'|^2 + 4\pi (x + y + u) \cdot x' - 2\pi (|x|^2 + |y + u|^2) \right)  dx'\\
		&\qquad = 2^{-d} \exp\left( \pi |x + y + u|^2 - 2\pi (|x|^2 + |y + u|^2) \right) \\
		&\qquad= 2^{-d} \exp\left( -\pi |x - y - u|^2 \right)  dx'.
	\end{align*}
	Thus, the expression is
\begin{align*}
	\exp&\left( -\pi |\xi - \eta|^2 \right) 2^{-d} \int  \exp\left( -\pi |x - y - u|^2 \right)\\
	&\qquad\times\, \left[ \int \exp\left( -\frac{\pi}{4} |r|^2 -2\pi i r \cdot \frac{\xi + \eta}{2}\right) m_t\left(u + \frac{r}{2}\right) m_t\left(u - \frac{r}{2}\right)dr \right]du,
\end{align*}
	as required.
\end{proof}
\begin{corollary}\label{Cor66-gg}
	For the Gabor matrix in \eqref{eq:GM-wave}  we have the estimate
	\[
	|\langle T_t \pi(x,\xi) g, \pi(y,\eta) g \rangle| \leq \sqrt{C_t} \,\,\exp(-\frac\pi 4|x-y|^2)\exp(-\frac\pi 2 |\xi-\eta|^2),\quad (x,\xi),(y,\eta)\in\rdd,
	\]
	where
	\[
	C_t = \|I\|_\infty \cdot vol(B_{2t}(0)) \cdot e^{4\pi t^2},\quad t>0.
	\] 
\end{corollary}
\begin{proof}
Let us work on the expression \eqref{eq:GM-wave}.
Let \(\alpha = x - y\), \(\beta = \xi - \eta\), and define
\begin{equation}\label{eq:Iu}
I(u) = \int_{\mathbb{R}^d} \exp\left( -\frac{\pi}{4} |r|^2 \right) m_t\left( u + \frac{r}{2} \right) m_t\left( u - \frac{r}{2} \right) \, dr.
\end{equation}
Noting that the phase factor has modulus 1,
\[
|\langle T_t \pi(z) g, \pi(w) g \rangle|^2 \leq 2^{-d} e^{-\pi |\beta|^2} \int_{\mathbb{R}^d} e^{-\pi |\alpha - u|^2} I(u) \, du.
\]
	We now compute $I(u)$ in \eqref{eq:Iu}	using the explicit densities of \(\mu_t\) in \eqref{eq:mt}.
\begin{itemize}
	\item [ Case \(d=1\).] Here \(m_t(y) = \frac{1}{2} \chi_{[-t,t]}(y)\). 
	The product \(m_t(u + r/2) m_t(u - r/2)\) is nonzero only if
	\[
	|u + r/2| \leq t \quad \text{and} \quad |u - r/2| \leq t.
	\]
	This implies
	\[
	|u| + \frac{|r|}{2} \leq t \quad \Rightarrow \quad |r| \leq 2(t - |u|)_+.
	\]
	Hence,
	\[
	I(u) = \int_{-(t-|u|)_+}^{(t-|u|)_+} e^{-\pi r^2 / 4} \cdot \frac{1}{4} \, dr
	= \frac{1}{2} \int_0^{2(t-|u|)_+} e^{-\pi s^2 / 4} \, ds
	= \sqrt{\pi} \cdot \operatorname{erf}\!\left( \sqrt{\pi} (t - |u|)_+ \right).
	\]
	Thus, \(I(u)\) is supported in \(|u| \leq t\), bounded by \(\sqrt{\pi}\), and smooth except at \(|u| = t\).
	\item [ Case \(d=2\).]  \(m_t(y) = \frac{1}{2\pi} (t^2 - |y|^2)_+^{-1/2}\).	
	This is the density of the uniform measure on the disk \(B_t(0)\). The function \(I(u)\) is the Gaussian-weighted autocorrelation of the indicator of the disk. By radial symmetry and known properties of the Funk–Hecke theorem or direct computation (see e.g., \cite[Chapter 4]{Rauch}), \(I(u)\) is bounded, supported in \(|u| \leq 2t\), and decays smoothly away from the origin.
	\item [ Case \(d=3\).] Here \(m_t(y) = \frac{1}{4\pi t} \, d\sigma_{\partial B_t(0)}\), the normalized surface measure on the sphere \(\partial B_t(0)\). Then
	\[
	I(u) = \left(\frac{1}{4\pi t}\right)^2 \int_{\mathbb{R}^3} e^{-\pi |r|^2/4} \left[ \int_{\partial B_t} \delta\!\left(u + \frac{r}{2} - y\right) d\sigma(y) \right] \left[ \int_{\partial B_t} \delta\!\left(u - \frac{r}{2} - z\right) d\sigma(z) \right] dr.
	\]
	This is the Gaussian convolution of the surface measure with itself, which is bounded, supported in \(|u| \leq 2t\), and smooth in the interior.
\end{itemize}	
 	In all cases \(d \leq 3\), we have:
	\[
	\sup_{u \in \mathbb{R}^d} |I(u)| \leq M < \infty, \qquad \operatorname{supp}(I) \subset B_{2t}(0).
	\]
	We can now majorize the Gabor matrix as follows:
	\[
	|\langle T_t \pi(z)g, \pi(w)g \rangle|^2  \leq e^{-\pi |\beta|^2} \cdot \|I\|_\infty \int_{|u| \leq 2t} e^{-\pi |\alpha - u|^2} \, du.
	\]
	By convexity, we obtain 
	\[
	|\alpha - u|^2 \geq \frac{1}{2} |\alpha|^2 - |u|^2 \geq \frac{1}{2} |\alpha|^2 - 4t^2 \quad \text{for } |u| \leq 2t,
	\]
	and thus
	\[
	e^{-\pi |\alpha - u|^2} \leq e^{-\frac{\pi}{2} |\alpha|^2} e^{4\pi t^2}.
	\]
	So
	\[
	\int_{|u| \leq 2t} e^{-\pi |\alpha - u|^2} \, du \leq e^{4\pi t^2} e^{-\frac{\pi}{2} |\alpha|^2} \cdot vol(B_{2t}(0)).
	\]
	Therefore,
	\[
	|\langle T_t \pi(z)g, \pi(w)g \rangle|^2 \leq C_t e^{-\pi |\beta|^2} e^{-\frac{\pi}{2} |\alpha|^2},
	\]
	with
	\[
	C_t = \|I\|_\infty \cdot vol(B_{2t}(0)) \cdot e^{4\pi t^2}.
	\]
This gives the claim.
	\end{proof}

Compared to \cite[Example 5.2.22]{Elena-book} we compute the decay in $\alpha=x-y$, uniformly with respect to the time variable $t\in\bR$.
\begin{figure}[t]
	\centering
	\includegraphics[width=\linewidth]{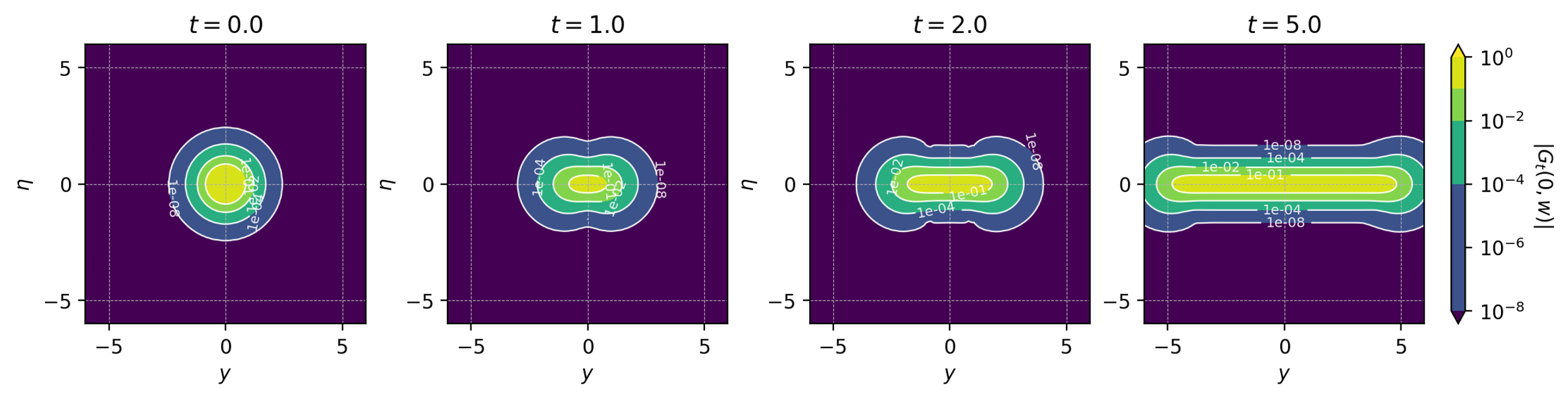}
	\caption{The magnitude of the Gabor matrix
		$G_t(0,w)=\langle T_t\pi(0,0)g,\pi(y,\eta)g\rangle$
		for the one-dimensional wave equation outlined for $y\in[-6,6]$, $\eta\in[-6,6]$ at times $t = 0,\,1,\,2,\, 5$, using a {\em logarithmic color scale} with \texttt{levels=[1e-8,1e-4,1e-2,1e-1,1]}.
		The early and intermediate stages of the one-dimensional wave evolution are captured: as $t$ increases, the spreading propagation typical of the solutions of wave equations manifests itself thorugh a stretching effect of the Gabor matrix.}
	\label{fig:gabor-wave-one-spot}
\end{figure}

	\subsection{Lacunas and ghost frequencies}
	A Petrovsky lacuna is an interior region where the fundamental solution of a linear hyperbolic partial differential equation vanishes, see \cite{ABG1970,ABG1973}. For the wave equation lacunas appear for every odd dimension $d$. The corresponding Wigner kernel eliminates lacunas, at least in part, because of the appearance there of the ghost frequencies, whereas the Gabor matrix re-establishes lacunas, in the sense that the matrix presents an asymptotic rapid decay in such regions. The case $d=1$ gives a simple evidence of the phenomena, with the warning that we have to refer to the propagator corresponding to initial data with $u_1=0$ in \eqref{LRintroA2}, namely the Fourier multiplier with symbol
	\begin{equation}
		\partial_t\sigma(t,\xi)=\cos(2\pi\xi t)=\frac 1 2 \big(e^{2\pi i\xi t}+e^{-2\pi i\xi t}\big)
	\end{equation}
	and kernel
	\begin{equation}
		E(t,x)=\cF^{-1}_{\xi\to x}\big(\cos(2\pi \xi t)\big)=\frac 1 2\big(\delta_t(x)+\delta_{-t}(x)\big)
	\end{equation}
	where $\delta_{\pm t}(x)$ are the Dirac distributions at the points $x=\pm t$, that is $\delta_{\pm t}(\varphi)=\f(\pm t)$, for $\f\in\mathcal{S}(\bR)$. The lacuna in this case is given by the interval $(-t,t)$. According to Theorem \ref{teo:3.3} we have
	\begin{equation}
		k_W(z,w)=\frac 1 2 \delta(\xi-\eta)W(\delta_t+\delta_{-t})(s,\xi), \qquad s=x-y, \, z=(x,\xi), \, w=(y,\eta).
	\end{equation}
	By applying \cite[Proposition 1.3.6 (iv)]{Elena-book}, we have
	\begin{align}
		W(\delta_t+\delta_{-t})(s,\xi)&=W\delta_t(s,\xi)+W\delta_{-t}(s,\xi)+2\Re(W(\delta_t,\delta_{-t}))(s,\xi)\\
		&=\delta_t(s)\otimes 1_\xi +\delta_{-t}(s)\otimes 1_\xi +\cos(4\pi \xi t)\delta(s)\otimes 1_\xi,
	\end{align}
	so that we conclude
	\begin{equation}
		k_W(z,w)=\frac 1 2 \delta(\xi-\eta) \big(
			\delta_t(s)\otimes 1_\xi +\delta_{-t}(s)\otimes 1_\xi +\cos(4\pi \xi t)\delta(s)\otimes 1_\xi
		\big).
	\end{equation}
	Note the appearance of the ghost term $\cos(4\pi \xi t)\delta(s)\otimes 1_\xi$ at $s=0$, center of the lacuna $|s|<t$. Instead, the Gabor matrix, provided by the short-time Fourier transform with Gaussian window, reads
	\begin{align}
		G_t(z,w)&=\exp\left(-\frac\pi 2 (\xi-\eta)^2\right)\int_{-\infty}^\infty \exp\left(-\frac \pi 2 (s-r)\right)E(t,r)dr\\
		&=\frac 1 2 \exp\left(-\frac\pi 2 (\xi-\eta)^2\right)\left(
		\exp\left(-\frac\pi 2 (s+t)^2\right)+\exp\left(-\frac\pi 2 (s-t)^2\right)
		\right), \qquad s=x-y,
	\end{align}
	where the ghost term at $s=0$ does not appear, being the Gabor matrix there
	\begin{equation}
		\exp\left(-\frac\pi 2 (\xi-\eta)^2\right)\exp\left(-\frac\pi 2 t^2\right).
	\end{equation}
	In dimension $d=3$, the propagation is given by the density measure $\frac 1 {4\pi t}d\sigma_{\partial B_t(0)}$ and the lacuna is given by the region $|s|<t$, $s\in \bR^3$. The estimates in Corollary \ref{Cor66-gg} can be improved as follows, with rapid decay with respect to $t\in\bR$ for $s=x-y$ in the lacuna.
	
	\begin{proposition}\label{thrm:wave3D}
		Let $T_t$ be the wave propagator in dimension $d = 3$. Then, the Gabor matrix in \eqref{eq:GM-wave} satisfies the estimate
\begin{equation}\label{eq:GMwave3d}
			|\langle T_t \pi(z) g, \pi(w) g \rangle|	\leq 
	t \, \exp\!\left(-\tfrac{\pi}{2}(|x - y| - t)^2\right)
	\exp\!\left(-\tfrac{\pi}{2} |\xi - \eta|^2\right),
\end{equation}
where $t>0, \, z = (x, \xi), \, w = (y, \eta) \in \mathbb{R}^{2d}$.
	\end{proposition}
	\begin{proof}		
From Proposition \ref{LRpropS1} and Remark \ref{rem38}, we have
\[
|\langle T_t \pi(z) g, \pi(w) g \rangle| = 
\exp\!\left(-\tfrac{\pi}{2} |\xi - \eta|^2\right)
\int e^{-\tfrac{\pi}{2} |s - r|^2} \, d\mu_t(r),
\]
with $s = x - y$ and $d\mu_t(r)=(4\pi t)^{-1}\,d\sigma_{\partial B_t(0)}(r)$, cf. \eqref{eq:mt}. We normalize the Gaussian as $g(t)=2^{3/4}e^{-\pi|y|^2}$, $y\in\bR^3$.

\medskip
Since $|s - r| \geq \big|\, |s| - |r| \,\big|$, we can estimate
\[
\int e^{-\tfrac{\pi}{2} |s - r|^2} \, d\mu_t(r)
\leq
\int e^{-\tfrac{\pi}{2} (|s| - |r|)^2} \, d\mu_t(r)
=
e^{-\tfrac{\pi}{2} (|s| - t)^2} \int d\mu_t(r),
\]
where $\int d\mu_t(r) = t$, and the conclusion follows. 	\end{proof}

\section{A metaplectic approach}\label{sec:Metap}
Metaplectic arguments have been used extensively for the analysis of Schrödinger equations with real quadratic Hamiltonians, for the propagator of the corresponding Cauchy problem can be represented as a one parameter subgroup of $\Mp(d,\bR)$. Even though the evolution of \eqref{e0} is not in $\Mp(d,\bR)$, metaplectic arguments can still be applied to simplify the computation of the Wigner operator \eqref{WignerOpHeat}, and there are situations, such as the complex Hermite equation, where they become indispensable. 

\subsection{A glimpse on the H\"ormander's metaplectic semigroup}
The complex symplectic group $\Sp(d,\bC)$ is the group of matrices $S\in\bC^{2d\times 2d}$ satisfying $S^\top JS=J$, where $J$ is as in \eqref{defJ}. Other than by $J$, it is generated by matrices in the form $\cD_E$ and $V_Q$, as in \eqref{defDE} and \eqref{defVQ}, respectively, with $E\in\GL(d,\bC)$ and $Q\in\Sym(d,\bC)$. Extending the metaplectic representation to $\Sp(d,\bC)$ so that for every $S\in\Sp(d,\bC)$, all the corresponding {\em metaplectic operators} $\hat S$ are bounded on $L^2(\rd)$ is not feasible. For the sake of concreteness, if $Q\in\Sym(d,\bC)$, the operator $\mathfrak{p}_Qf(x)=e^{i\pi Qx\cdot x}f(x)$ is bounded on $L^2(\rd)$ if and only if $\Im(Q)\geq0$. Therefore, an analogue of the metaplectic representation for complex symplectic matrices needs to restrict $\Sp(d,\bR)$ to a subsemigroup of it, thereby associating to every {\em admissible} $S\in\Sp(d,\bC)$ an operator bounded on $L^2(\rd)$, defined up to a phase factor. A sufficient positivity condition on these symplectic matrices for the $L^2$ boundedness of the corresponding operators, together with the form of the corresponding {\em metaplectic operators}, was established by H\"ormander in \cite{Hormander1995}. We consider here an equivalent formulation proved in the same paper.
\begin{definition}\label{defHSG}
	We say that an operator $\hat S:L^2(\rd)\to L^2(\rd)$ is in the {\em metaplectic semigroup}, here denoted by $\Mp_+(d,\bC)$, if $\hat S=\hat S_1\hat T\hat S_2$, where $\hat S_1,\hat S_2\in\Mp(d,\bR)$ and 
	\begin{equation}
		\hat T=\hat T_1\otimes\ldots\otimes \hat T_d,
	\end{equation}
	where, for $j=1,\ldots,d$, $\hat T_j$ has one of the following forms
	\begin{align}
		&\hat T_jg(t)=\mathfrak{p}_{i\alpha_j}g(t)=e^{-\pi \alpha_j t^2}g(t), \qquad \text{for some $\alpha_j\geq0$}, 
	\end{align}
	or
		\begin{align}
			\widehat{T}_j g(x)
			= \mathfrak{R}_{\vartheta_j} g(x)
			&= \frac{1}{\cosh(\vartheta_j)^{1/2}}
			\int_{-\infty}^{\infty}
			\hat{g}(\eta) \,
			\exp\!\left(
			-\frac{\pi}{\cosh(\vartheta_j)}
			\big[ (x^2 + \eta^2)\sinh(\vartheta_j) - 2 i x \eta \big]
			\right)
			\, d\eta,
		\end{align}
		for some  $\vartheta_j \ge 0$ ($g \in L^2(\mathbb{R})$).
\end{definition}
For the rest of this work, if $\hat S\in \Mp_+(d,\bC)$, the decomposition $\hat S=\hat S_1\hat T\hat S_2$ of Definition \ref{defHSG} is understood.
\begin{remark}
	Observe that if $\alpha_j=\vartheta_j=0$ for some $j=1,\ldots,d$ in Definition \ref{defHSG}, the two operators defining $\hat T_j$ have the same expression: $\hat T_j=id_{L^2}$.
\end{remark}
For the properties of $\Mp_+(d,\bC)$ we address \cite{Hormander1995}. In particular, it is proven in \cite{Hormander1995} that the double cover $\pi^{Mp}:\Mp(d,\bR)\to\Sp(d,\bR)$ extends to a double cover of a subsemigroup of $\Sp(d,\bC)$ constructed as follows. For $\alpha,\vartheta\geq0$, H\"ormander considered
\begin{equation}
	V_{i\alpha}=\begin{pmatrix}
		1 & 0\\
		i\alpha & 1
	\end{pmatrix} \qquad and \qquad
	\mathcal{R}_{\vartheta}=\begin{pmatrix}
		\cosh\vartheta & -i\sinh\vartheta\\
		i\sinh\vartheta & \cosh\vartheta
	\end{pmatrix},
\end{equation}
thereby defining $T=\pi^{Mp}_+(\hat T)=T_1\otimes\ldots\otimes T_d$, where $T_j=V_{i\alpha_j}$ or $T_j=\mathcal{R}_{\vartheta_j}$, correspondingly. Consequently, $\pi_{+}^{Mp}(\hat S)=S_1TS_2$, with $ S_j=\pi^{Mp}(\hat S_j)$ and $T$ as above.
\begin{definition}
	We set $\Sp_+(d,\bC)=\pi^{Mp}_+(\Mp_+(d,\bC))$.
\end{definition}
It was proven by H\"ormander that $\Sp_+(d,\bC)\subset\Sp(d,\bC)$ is a subsemigroup.
The tensor nature of the operator $\hat T$ in Definition \ref{defHSG} suggests that a deeper study of tensor products of operators in $\Mp_+(d,\bC)$ is needed. The following result generalizes Theorem B1 in \cite{CG2023} to H\"ormander's framework.
	
	\begin{lemma}\label{lemmatensor}
		Let $\hat S_1,\hat S_2\in\Mp_+(d,\bC)$. There exists a unique operator $\hat S\in\Mp_+(2d,\bC)$ such that
		\begin{equation}
			\hat S(f\otimes g)=\hat S_1f\otimes \hat S_2g, \qquad f,g\in L^2(\rd).
		\end{equation}
		Moreover, if $S_j=\pi^{Mp}_+(\hat S_j)$, $j=1,2$, then $\pi^{Mp}_+(\hat S)=S_1\otimes S_2$.
	\end{lemma}
	\begin{proof}
		Since $\Mp_+(d,\bC)\subseteq B(L^2(\rd))$, the existence and the uniqueness of $\hat S:=\hat S_1\otimes \hat S_2\in B(L^2(\rd))$ follow by standard operator theory, see e.g. \cite[Theorem 2.6.12]{Kadison}. We have to prove that $\hat S\in\Mp_+(d,\bC)$ and that its projection is $S=S_1\otimes S_2$. 
		
		Let $\hat S_1^{(1)},\hat S_1^{(2)},\hat S_2^{(1)},\hat S_2^{(2)}\in\Mp(d,\bR)$ and $\hat T_1,\hat T_2$ be such that
		\begin{equation}
			\hat S_j=\hat S_j^{(1)}\hat T_j\hat S_j^{(2)}, \qquad j=1,2.
		\end{equation}
		Then,
		\begin{equation}
			\hat S=\hat S_1 \otimes \hat S_2=(\hat S_1^{(1)}\hat T_1\hat S_1^{(2)})\otimes(\hat S_2^{(1)}\hat T_2\hat S_2^{(2)}).
		\end{equation}
		Since $\hat S$ and
		\begin{equation}
			(\hat S_1^{(1)}\otimes \hat S_2^{(1)})(\hat T_1\otimes \hat T_2)(\hat S_1^{(2)}\otimes\hat S_2^{(2)})
		\end{equation}
		are bounded on $L^2(\rd)$ and they act the same on tensors $f\otimes g$, $f,g\in L^2(\rd)$,
		\begin{equation}
			\hat Sf=(\hat S_1^{(1)}\otimes \hat S_2^{(1)})(\hat T_1\otimes \hat T_2)(\hat S_1^{(2)}\otimes\hat S_2^{(2)})f,\qquad f\in L^2(\rdd),
		\end{equation}
		by uniqueness of tensor products of bounded operators on $L^2(\rd)$.
		
		By \cite[Theorem B1]{CG2023}, $\hat S_1^{(1)}\otimes \hat S_2^{(1)},\hat S_1^{(2)}\otimes\hat S_2^{(2)}$ are in $\Mp(2d,\bR)$ with projections $S_1^{(1)}\otimes S_2^{(1)}$ and $S_1^{(2)}\otimes S_2^{(2)}$, respectively. On the other hand, the same claim holds trivially for $\hat T_1\otimes \hat T_2$ by construction, for $\hat T_1$ and $\hat T_2$ are themselves tensor products and $\pi^{Mp}_+(\hat T_1\otimes \hat T_2)=T_1\otimes T_2$ by definition. Consequently, $\hat S\in \Mp_+(2d,\bC)$ and 
		\begin{align*}
		\pi^{Mp}_+(\hat S)&=\pi^{Mp}_+(\hat S_1^{(1)}\otimes \hat S_2^{(1)})\pi^{Mp}_+(\hat T_1\otimes \hat T_2)\pi^{Mp}_+(\hat S_1^{(2)}\otimes\hat S_2^{(2)})=(S_1^{(1)}\otimes S_2^{(1)})(T_1\otimes T_2)(S_1^{(2)}\otimes S_2^{(2)})\\
		&=(S_1^{(1)}T_1 S_1^{(2)})\otimes( S_2^{(1)} T_2 S_2^{(2)})=S_1\otimes S_2.
		\end{align*}
		This concludes the proof.
	\end{proof}

\begin{example}
	Let $Q\in\Sym(d,\bC)$ with $\Im(Q)\geq0$. The operator
	\begin{equation}\label{defpQcomp}
		\mathfrak{p}_Qf(x)=\Phi_Q(x)f(x),
	\end{equation}
	where $\Phi_Q(x)=e^{i\pi Qx\cdot x}$ is a complex Gaussian, is metaplectic with projection $\pi^{Mp}_+(\mathfrak{p}_Q)=V_Q$, where $V_Q$ is defined precisely as in \eqref{defVQ}. Moreover, if $P\in\Sym(d,\bC)$ with $\Im(P)\leq0$, the Fourier multiplier
	\begin{equation}\label{convgaussop}
		\mathfrak{m}_Pf=\cF^{-1}(\Phi_{-P}\hat f)
	\end{equation}
	is metaplectic with projection $V_P^\top$.
\end{example}

\subsection{Metaplectic analysis of the complex heat equation}	

A metaplectic operator $\hat S$ in $\Mp(d,\bR)$ satisfies
	\begin{equation}\label{gg161001}
		W(\hat Sf,\hat Sg)=\mathfrak{T}_{S^{-1}}W(f,g),
	\end{equation}
	where $\mathfrak{T}_{S^{-1}}f(z)=f(S^{-1}z)$ is a linear rescaling, thus the Wigner operator of $\hat S$ consists of a rescaling of the phase-space domain, and, consequently, its Wigner kernel $k_W$ is the distribution $$k_W(z,w)=\delta_{w=S^{-1}z},$$ which is concentrated on the manifold $\{z=Sw\}$. The propagator of the Cauchy problem \eqref{e0}, which defines the one parameter semigroup $\{\hat S_t\}_{t\geq0}\subset\Mp_+(d,\bC)$ expressed by \eqref{e1} does not satisfy this property: by decomposing
	\begin{equation}
		\hat S_t=\hat S_1\hat T,
	\end{equation}
	where $\hat S_1=\mathfrak{m}_{-4\pi \beta t I}\in \Mp(d,\bR)$ and $\hat T=\mathfrak{m}_{4\pi i\alpha t I}\in\Mp_+(d,\bC)$, we obtain by \eqref{WignerOpHeat} that 
	\begin{equation}
		W(\hat Sf,\hat Sg)(z)=K' W(f,g)(S_1^{-1}z),
	\end{equation}
	where the contribution of $K'$ cannot be shifted onto the phase-space variable $z$. Above, we are omitting the dependence on $t$ of $\hat S_1$ and $\hat T$ to simplify the notation. To get a deeper insight, we begin by exhibiting the Wigner operator of Gaussian multipliers
	\begin{equation}
		\mathfrak{m}_{-iP}f(x)=\cF^{-1}(\f_P\hat f)(x), \qquad P\geq0,\quad  P\in\bR^{d\times d} \,\mbox{ symmetric},
	\end{equation}
	where $\f_P(u)=e^{-\pi Pu\cdot u}$, using metaplectic arguments.
	
	\begin{proposition}\label{propIntertWig1}
		For every real symmetric matrix $P\geq0$ and every $f,g\in L^2(\rd)$,
		\begin{equation}\label{intertWmiP}
			W(\mathfrak{m}_{-iP}f,\mathfrak{m}_{-iP}g)=(\mathfrak{m}_{-iP/2}\otimes\mathfrak{p}_{2iP})W(f,g),
		\end{equation}
		up to a phase factor, where $\mathfrak{p}_{2iP}$ is defined as in \eqref{defpQcomp}.
	\end{proposition}
	\begin{proof}
		Let $f,g\in L^2(\rd)$. Observe that
		\begin{equation}
			\overline{\mathfrak{m}_{-iP}g}=\mathfrak{m}_{-iP}\bar g.
		\end{equation}
		Let $\widehat{\cA_{1/2}}\in\Mp(2d,\bR)$ be the metaplectic operator so that $W(f,g)=\widehat{\cA_{1/2}}(f\otimes \bar g)$, specifically $\widehat{\cA_{1/2}}=\cF_2\mathfrak{T}_{E_w}$, where $\cF_2$ is defined as in \eqref{defF2}, and $\mathfrak{T}_{E_w}$ is defined as in \eqref{defTE} with
		\begin{equation}
			E_w=\begin{pmatrix}
				I & I/2\\
				I & -I/2
			\end{pmatrix}.
		\end{equation}
		Explicitly, the projection of $\widehat{\cA_{1/2}}$ is
		\begin{equation}
			\cA_{1/2}=\begin{pmatrix}
				\frac I 2 & \frac I 2 & O & O\\
				O & O & \frac I 2 & -\frac I 2\\
				O & O & I & I\\
				-I & I & O & O
			\end{pmatrix},
		\end{equation}
		see also \cite{CGRPartII} for the details. Then,
		\begin{align}
			W(\mathfrak{m}_{-iP}f,\mathfrak{m}_{-iP}g)&=\widehat{\cA_{1/2}}(\mathfrak{m}_{-iP}\otimes \mathfrak{m}_{-iP})(f\otimes\bar g)=\widehat{\cA_{1/2}}(\mathfrak{m}_{-iP}\otimes \mathfrak{m}_{-iP})\widehat{\cA_{1/2}}^{-1}\widehat{\cA_{1/2}}(f\otimes\bar g)\\
			&=\widehat{\cA_{1/2}}(\mathfrak{m}_{-iP}\otimes \mathfrak{m}_{-iP})\widehat{\cA_{1/2}}^{-1}W(f,g).
		\end{align}
		By arguing at the level of the symplectic projections, we obtain
		\begin{align}
			\cA_{1/2}(V_{-iP}^\top\otimes V_{-iP}^\top)\cA_{1/2}^{-1}&=\begin{pmatrix}
				\frac I 2 & \frac I 2 & O & O\\
				O & O & \frac I 2 & -\frac I 2\\
				O & O & I & I\\
				-I & I & O & O
			\end{pmatrix}
			\begin{pmatrix}
				I & O & -iP & O\\
				O & I & O & -iP\\
				O & O & I & O\\
				O & O & O & I
			\end{pmatrix}
			\begin{pmatrix}
				I & O & O & -\frac I 2\\
				I & O & O & \frac I 2\\
				O & I & \frac I 2 & O\\
				O & -I & \frac I 2 & O
			\end{pmatrix}\\
			&=\begin{pmatrix}
				I & O & -iP/2 & O\\
				O & I & O & O\\
				O & O & I & O\\
				O & 2iP & O & I
			\end{pmatrix}=V_{-iP/2}^\top\otimes V_{2iP}.
		\end{align}
		Therefore, \eqref{intertWmiP} follows up to a phase factor by the fact that the double covering $\pi_+^{Mp}$ is a homomorphism.
		
	\end{proof}
	
	As a corollary, we retrieve \eqref{WignerOpHeat}. 
	\begin{corollary}
		Consider the propagator $\widehat{S_t}\in\Mp_+(d,\bC)$ of \eqref{e0}, expressed in \eqref{e1}. Then, for every $t\geq0$ and $x,\xi\in\rd$,
		\begin{equation}
			W(u(t,\cdot))(x,\xi)=e^{-8\pi^2\alpha t|\xi|^2}\cF^{-1}(e^{-2\pi^2\alpha t|\cdot|^2}\otimes 1)\ast Wu_0(x+4\pi\beta t\xi,\xi).
		\end{equation}
		In particular, if $t>0$,
		\begin{equation}
			W(u(t,\cdot))(x,\xi)=(2\pi\alpha t)^{-d/2}e^{-8\pi^2\alpha t|\xi|^2}[(e^{-|\cdot|^2/2\alpha t}\otimes\delta)\ast Wu_0](x+4\pi\beta t\xi,\xi).
		\end{equation}
	\end{corollary}
	\begin{proof}
		For $t=0$ there is nothing to say. Let $t>0$. A simple restatement of \eqref{e1}, gives
	\begin{equation}\label{propgausscauchy}
		u(t,x)=\mathfrak{m}_{-4\pi i(\alpha+i\beta)t I}u_0(x)=\mathfrak{m}_{-4\pi i\alpha t I}\mathfrak{m}_{4\pi \beta t I}u_0(x),
	\end{equation}
	yielding
		\begin{align}
			W(u(t,\cdot))(x,\xi)&=W(\widehat{S_t}u_0)(x,\xi)=W(\mathfrak{m}_{4\pi \beta t I}\mathfrak{m}_{-4\pi i\alpha t I}u_0)(x,\xi)\\
			&=W(\mathfrak{m}_{-4\pi i\alpha t I}u_0)(x+4\pi\beta t\xi,\xi).
		\end{align}
		By Proposition \ref{propIntertWig1} with $P=4\pi\alpha tI$, for $y,\eta\in\rd$,
		\begin{align}
			W(\mathfrak{m}_{-4\pi i\alpha t I}u_0)(y,\eta)&=(\mathfrak{m}_{-2\pi i\alpha tI}\otimes\mathfrak{p}_{8\pi i\alpha tI})Wu_0(y,\eta)\\
			&=e^{-8\pi^2\alpha t|\eta|^2}\cF^{-1}(e^{-2\pi^2\alpha t|\cdot|^2}\otimes 1)\ast Wu_0(y,\eta),
		\end{align}
		whence
		\begin{align}
			W(u(t,\cdot))(x,\xi)&=W(\mathfrak{m}_{-4\pi i\alpha t I}u_0)(x+4\pi\beta t\xi,\xi)\\
			&=e^{-8\pi^2\alpha t|\xi|^2}\cF^{-1}(e^{-2\pi^2\alpha t|\cdot|^2}\otimes 1)\ast Wu_0(x+4\pi\beta t\xi,\xi)\\
			&=(2\pi\alpha t)^{-d/2}e^{-8\pi^2\alpha t|\xi|^2}[(e^{-|\cdot|^2/2\alpha t}\otimes\delta)\ast Wu_0](x+4\pi\beta t\xi,\xi).
		\end{align}
		This concludes the proof.
	\end{proof}
	
	In conclusion, metaplectic arguments simplified the computation of the Wigner operator for the complex heat semigroup, as the burden computation of Section \ref{sec:ComplHeatEq} reduced basically to a matrix multiplication. Then, we turn our attention to the Hermite equations.
	
	\subsection{The Hermite semigroup}
	Consider the Cauchy problem
	\begin{equation}\label{HermiteCompl}
		\begin{cases}
			\partial_tu=\Big(\frac{1}{2\pi}\Delta-2\pi|x|^2\Big)u,\\
			u(0,x)=u_0,
		\end{cases}
	\end{equation}
	with $u_0\in\cS'(\rd)$, whose evolution is expressed by the one-parameter semigroup
	\begin{equation}\label{propHermite}
		\mathfrak{R}_{2t}u_0(x)=\frac{1}{\cosh(2t)^{d/2}}\int_{\rd}\widehat{u_0}(\eta)e^{-\pi \tanh(2t)(|x|^2+|\eta|^2)}e^{2\pi ix\eta/\cosh(2t)}d\eta,
	\end{equation}
	see \cite[Chapter 5.2]{Folland}. In \eqref{propHermite}, we are adopting the abuse of notation $\mathfrak{R}_{2t}=\mathfrak{R}_{\Theta}$, where $\Theta=(2t,\ldots,2t)\in\rd_{\geq0}$. In this section, we will therefore study the Wigner kernel and the Gabor matrix of the operators $\mathfrak{R}_\Theta$ in full generality.
	
	\begin{remark}
		Our interest in the full-general setting is also justified by the fact that the operator $\mathfrak{R}_\Theta$, with $\Theta=(\vartheta_1,\ldots,\vartheta_d)t$, for $t,\vartheta_1,\ldots,\vartheta_d\geq0$, defines the propagator for the anisotropic Hermite equation
		\begin{equation}
		\begin{cases}
			\partial_tu=\sum_{j=1}^d\Big(\frac{\vartheta_j}{4\pi}\frac{\partial^2}{\partial{x_j^2}}-\pi\vartheta_j x_j^2 \Big) u,\\
			u(0,x)=u_0(x).
		\end{cases}
	\end{equation}
	\end{remark} 
	
	\begin{proposition}\label{propIntertWig2}
			For $\Theta=(\vartheta_1,\ldots,\vartheta_j)\in\rd_{\geq0}$, we have
			\begin{align}
				W(\mathfrak{R}_\Theta f,\mathfrak{R}_\Theta g)&=
				\label{gg1-3}
				\mathfrak{T}_{\sqrt 2 I_{2d}}\mathfrak{R}_{(\Theta,\Theta)}\mathfrak{T}_{1/\sqrt 2 I_{2d}}W(f,g), \qquad f,g\in L^2(\rd)
			\end{align}
			up to a phase factor, where the rescalings are defined as in \eqref{defTE}.
		\end{proposition}
		\begin{proof}
			Let $f,g\in L^2(\rd)$. First, observe that if $\vartheta\geq0$ then
			\begin{equation}
				\overline{\mathfrak{R}_\vartheta g}=\mathfrak{R}_\vartheta\bar g.
			\end{equation}
			Since $\mathfrak{R}_\Theta=\bigotimes_{j=1}^d\mathfrak{R}_{\vartheta_j}$, we obtain that
			\begin{equation}\label{tensconjR}
				\overline{\mathfrak{R}_\Theta g}=\mathfrak{R}_\Theta\bar g,
			\end{equation}
			holds whenever $g=\bigotimes_{j=1}^dg_j$, $g_j\in L^2(\bR)$, $j=1,\ldots,d$, and therefore a standard density argument shows that \eqref{tensconjR} holds for every $g\in L^2(\rd)$.
			Arguing similarly to Proposition \ref{propIntertWig1}, 
			\begin{align}
			W(\mathfrak{R}_{\Theta}f,\mathfrak{R}_{\Theta}g)&=\widehat{\cA_{1/2}}(\mathfrak{R}_{\Theta}\otimes \mathfrak{R}_{\Theta})(f\otimes\bar g)=\widehat{\cA_{1/2}}(\mathfrak{R}_{\Theta}\otimes \mathfrak{R}_{\Theta})\widehat{\cA_{1/2}}^{-1}\widehat{\cA_{1/2}}(f\otimes\bar g)\\
			&=\widehat{\cA_{1/2}}(\mathfrak{R}_{\Theta}\otimes \mathfrak{R}_{\Theta})\widehat{\cA_{1/2}}^{-1}W(f,g).
		\end{align}
		At the level of $\Sp_+(d,\bC)$,
			\begin{align}
			\cA_{1/2}(\mathcal{R}_{\Theta}\otimes \mathcal{R}_{\Theta})\cA_{1/2}^{-1}&=\left(\begin{array}{ccc|ccc}
			\cosh(\vartheta_1) & \ldots  & 0 & -i\sinh(\vartheta_1)/2 & \ldots & 0\\
			\vdots & \ddots  & \vdots & \vdots & \ddots & \vdots\\
			0 & \ldots & \cosh(\vartheta_d)  & 0 & \ldots & -i\sinh(\vartheta_d)/2\\
			\hline
			2i\sinh(\vartheta_1) & \ldots & 0 & \cosh(\vartheta_1) & \ldots & 0\\
			\vdots & \ddots  & \vdots & \vdots & \ddots & \vdots\\
			0 & \ldots & 2i\sinh(\vartheta_d) & 0 & \ldots &  \cosh(\vartheta_d)
		\end{array}\right)\\
			&=\cD_{\sqrt 2 I_{2d}}(\mathcal{R}_\Theta\otimes \mathcal{R}_\Theta)\cD_{1/\sqrt 2 I_{2d}},
		\end{align}
		Formula \eqref{gg1-3} follows by observing $\cR_{\Theta}\otimes\cR_\Theta=\cR_{(\Theta,\Theta)}$.
		\end{proof}

		Recall that if $T_1:\cS'(\rd)\to\cS'(\rd)$ and $T_2:\cS'(\rd)\to\cS'(\rd)$ have kernels $k_{T_1}\in\cS'(\rdd)$ and $k_{T_2}\in\cS'(\rdd)$, the composition $T_1T_2$ has kernel $k_{T_1T_2}$ given by
		\begin{equation}\label{compositionLaw}
			k_{T_1T_2}(x,y)=\int_{\rd}k_{T_1}(x,z)k_{T_2}(z,y)dz,
		\end{equation}
		(see  Theorem 4.6 in \cite{CGP2025}). By writing explicitly the Fourier transform in the definition of $\mathfrak{R}_\Theta$, $\Theta=(\vartheta_1,\ldots,\vartheta_d)\in\rd_{\geq0}$, and integrating the Gaussian, we easily see that the kernel $k'$ of $\mathfrak{R}_{(\Theta,\Theta)}$ is given by
		\begin{align}\label{defkprime}
			k'(x, \xi, y, \eta)
			= \underset{\vartheta_j\neq0}{\prod_{j=1}^{d}} &[\sinh(\vartheta_j)]^{-1} \,
			\exp\!\left(
			-\frac{\pi}{\tanh(\vartheta_j)}
			\sum_{j=1}^{d} \big( x_j^2 + \xi_j^2 + y_j^2 + \eta_j^2 \big)
			\right)\\
		&\times	\exp\!\left(
			\frac{2\pi}{\sinh(\vartheta_j)}
			\sum_{j=1}^{d} (x_j y_j + \xi_j \eta_j)
			\right)\underset{\vartheta_j=0}{\prod_{j=1}^d}\delta_{x_j=y_j}\delta_{\xi_j=\eta_j}.
		\end{align}
		
		A straightforward application of \eqref{compositionLaw} on \eqref{defkprime}, using that the kernels of $\mathfrak{T}_{2^{\pm1/2}I}$ are $$K_{\pm}(z,w)=2^{\pm d/2}\delta_{w=2^{\pm1/2}z},$$ yields the following corollary.
		\begin{corollary}
			The Wigner kernel $k_W$ of $\mathfrak{R}_\Theta$ is given by
		\begin{align}\label{defkprime2}
			k_W(x, \xi, y, \eta)
			= \underset{\vartheta_j\neq0}{\prod_{j=1}^{d}} &[\sinh(\vartheta_j)]^{-1} \,
			\exp\!\left(
			-\frac{2\pi}{\tanh(\vartheta_j)}
			\sum_{j=1}^{d} \big( x_j^2 + \xi_j^2 + y_j^2 + \eta_j^2 \big)
			\right)\\
			&\times\exp\!\left(
			\frac{4\pi}{\sinh(\vartheta_j)}
			\sum_{j=1}^{d} (x_j y_j + \xi_j \eta_j)
			\right)\underset{\vartheta_j=0}{\prod_{j=1}^d}\delta_{x_j=y_j}\delta_{\xi_j=\eta_j}.
		\end{align}
		
		\end{corollary}
					
		
		\begin{proposition}\label{corHermite}
			The modulus  of the Gabor matrix $G(z,w)$ of $\mathfrak{R}_\Theta$ is
			\begin{align}\label{GabMatHermite}
				|G(z,w)|
				= 2^{-d} \underset{\vartheta_j\neq0}{\prod_{j=1}^d} & \sinh(\vartheta_j)^{-1/2} \,
				\exp\!\left(
				-\frac{\pi}{4} \sum_{j=1}^d
				\big[ 1 + e^{-\vartheta_j} \big] (z_j - w_j)^2
				\right) \\
				&\quad \times
				\exp\!\left(
				-\frac{\pi}{4} \sum_{j=1}^d
				\big[ 1 - e^{-\vartheta_j} \big] (z_j + w_j)^2
				\right)\underset{\vartheta_j=0}{\prod_{j=1}^d}\exp\!\left(-\frac{\pi}{2}(z_j-w_j)^2\right).
			\end{align}
			
			In particular, we obtain the modulus of the Gabor matrix $G_t(z,w)$ of the evolution operator of \eqref{HermiteCompl}:
		\begin{align*}
			|G_t(z,w)| =
			2^{-d} &[\sinh(2t)]^{-d/2}\exp\!\left[
			-\frac{\pi}{4} \Big(
			(1 + e^{-t})|z - w|^2
			+ (1 - e^{-t})|z + w|^2
			\Big)
			\right].
		\end{align*}
		\end{proposition}
		\begin{proof}
			It is a straightforward computation accompanied by a standard tensorization argument, using \eqref{LReqS4}. By the expression \eqref{defkprime2} it is clear that we just need to argue for each $j$ separately. Therefore, for each $j=1,\ldots,d$, we convolve the corresponding factor of \eqref{defkprime2} with the Gaussian $2e^{-2\pi|\cdot|^2}$. The identity \eqref{GabMatHermite} is then obtained by taking the product over $j$, and observing that $\cosh(\vartheta_j)-\sinh(\vartheta_j)=\exp(-\vartheta_j)$.
			
		\end{proof}
%
		
		\section{The complex Hermite equation}\label{sec:8}
		Finally, we consider the Cauchy problem
	\begin{equation}\label{HermiteComplCompl}
		\begin{cases}
			\partial_tu=(\vartheta+i\mu)\Big(\frac{1}{4\pi}\Delta-\pi|x|^2\Big)u,\\
			u(0,x)=u_0\in\cS(\rd),
		\end{cases}
	\end{equation}	
	with $\vartheta>0$ and $\mu\in\bR$. With an abuse of notation, the propagator is given by the operator
	\begin{align}
		\mathfrak{R}_{(\vartheta+i\mu)t}u_0(x)=\cosh((\vartheta+i\mu)t)^{-d/2}\int_{\rd}\widehat{u_0}(\eta)e^{-\pi \tanh((\vartheta+i\mu)t)(|x|^2+|\eta|^2)}e^{2\pi ix\eta/\cosh((\vartheta+i\mu)t)}d\eta.
	\end{align}
	To compute $\mathfrak{R}_{(\vartheta+i\mu)t}$ explicitly, we work at the level of metaplectic projections. Specifically, using that
	\begin{equation}
		\begin{pmatrix}
			\cosh((\vartheta+i\mu)t)I & -i\sinh((\vartheta+i\mu)t)I\\
			i\sinh((\vartheta+i\mu)t)I & \cosh((\vartheta+i\mu)t)I
		\end{pmatrix}=
		\underbrace{\begin{pmatrix}
			\cosh(\vartheta t)I & -i\sinh(\vartheta t)I\\
			i\sinh(\vartheta t)I & \cosh( \vartheta t)I
		\end{pmatrix}}_{=:\mathcal{R}_{\vartheta t}}
		\underbrace{\begin{pmatrix}
			\cos(\mu t)I & \sin(\mu t)I\\
			-\sin(\mu t)I & \cos( \mu t)I
		\end{pmatrix}}_{=S_{\mu t}},
	\end{equation}
	we obtain
	\begin{equation}
		\mathfrak{R}_{(\vartheta+i\mu)t}=\mathfrak{R}_{\vartheta t}\cF_{\mu t}\in {\Mp}_+(d,\bC),
	\end{equation}
	where $\cF_{\mu t}$ denotes the Fractional Fourier transform with parameter $\mu t$:
	\begin{equation}
		\cF_{\mu t}f(\xi)=(1-i\cot(\mu t))^{d/2}\int_{\rd}f(x)e^{i\pi(|x|^2+|\xi|^2)/\tan(\mu t)}e^{-2\pi ix\xi/\sin(\mu t)}dx, \qquad f\in\cS(\rd)
	\end{equation}
	(if $\sin(\mu t)\neq0$). We omit the correct choice of the branch of the square root for simplicity. Moreover $\cF_{2k\pi}=id_{L^2}$ and $\cF_{\pi+2k\pi}=\mathfrak{T}_{-I}$, $k\in\mathbb{Z}$, up to a phase factor.
	Observe that if $g(t)=e^{-\pi|t|^2}$, $\cF_{\mu t}g=g$ for every $\mu\in\bR$ and for every $t\geq0$. Thus, the Gabor matrix of the propagator \eqref{HermiteComplCompl} is given by
	\begin{align}
		\la \mathfrak{R}_{(\vartheta+i\mu)t}\pi(z)g,\pi(w)g \ra&= \la \mathfrak{R}_{\vartheta t}\cF_{\mu t}\pi(z)g,\pi(w)g \ra= \la \mathfrak{R}_{\vartheta t}\pi(S_{\mu t}z)\cF_{\mu t}g,\pi(w)g \ra\\
		&=\la \mathfrak{R}_{\vartheta t}\pi(S_{\mu t}z)g,\pi(w)g \ra.
	\end{align}
	By applying Proposition \ref{corHermite}, we obtain:
	\begin{corollary}\label{cor:8.1}
		The Gabor matrix of the evolution operator of \eqref{HermiteComplCompl} satisfies for $t>0$
		\begin{align}
			\big|\langle \mathfrak{R}_{(\vartheta + i\mu)t} \pi(z)g, \pi(w)g \rangle\big|
			= 2^{-d} \, &[\sinh(\vartheta t)]^{-d/2} \,
			\exp\!\left(
			-\frac{\pi}{4} ( 1 + e^{-\vartheta t}) 
			\, |S_{\mu t} z - w|^2
			\right) \\
			&\quad \times
			\exp\!\left(
			-\frac{\pi}{4} (1 -e^{-\vartheta t})
			\, |S_{\mu t} z + w|^2
			\right).
		\end{align}
	\end{corollary}
	See Figure \ref{fig:gabor-complex-hermite-rotation} for a visual representation. As for the Wigner kernel $k_W$, we use the composition law of kernels \eqref{compositionLaw}.
	
\begin{figure}
	\centering
	\includegraphics[width=\textwidth]{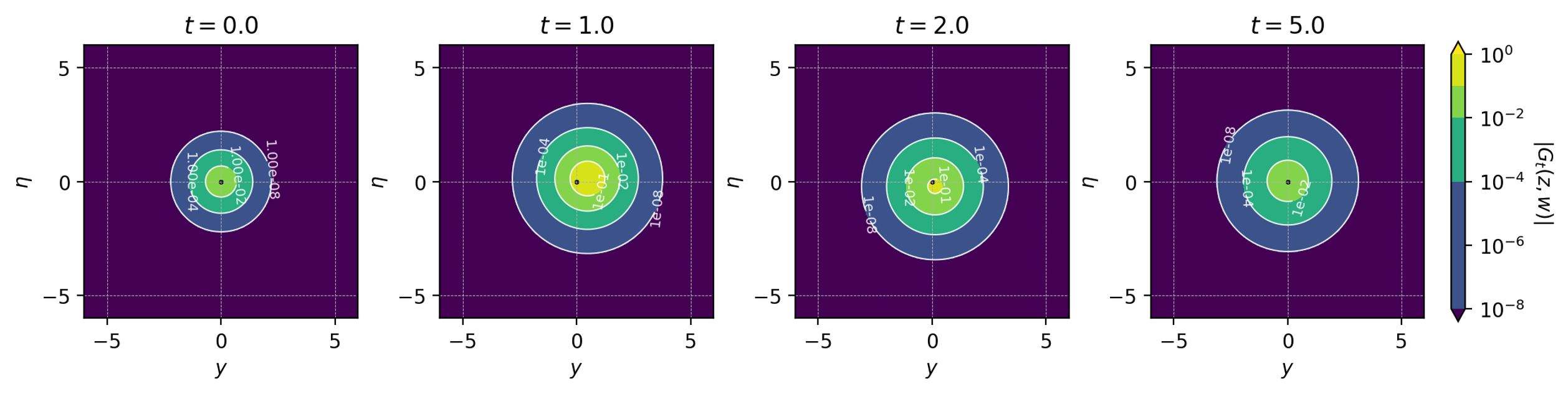}
	\caption{
		Plots in logarithmic color scale of 
		$|G_t(z,w)| = \big|\langle R_{(\vartheta+i\mu)t}\pi(z)g,\,\pi(w)g\rangle\big|$
		for the complex Hermite equation with random parameters $\vartheta=0.7$ and $\mu=1.3$, outlined for $z=(0,1)$. 
		The panels correspond to times $t = 0\,1,\,2,\,5$. The peak of the Gabor matrix undergoes a clear
		\emph{rotation} in the $(y,\eta)$-plane induced by the fractional Fourier transform
		component $\cF_{\mu t}$, its Gaussian shape and width remain constant, and it undergoes an exponential dumping due to the coefficient $(2\sinh(\vartheta t))^{-1/2}$. Hence, unlike the complex heat equation displayed in Figure~\ref{fig:gabor-complex-heat}, there is no diffusion or spreading, but a rigid rotation of the localization region in phase space.}
	\label{fig:gabor-complex-hermite-rotation}
\end{figure}

	\begin{corollary}\label{cor:8.2}
		For $t>0$, the Wigner kernel $k_W$ of the evolution operator of \eqref{HermiteComplCompl} is
		\begin{equation}
			k_W(z,w)
			= [\sinh(\vartheta t)]^{-d} \,
			\exp\!\left(
			-\frac{2\pi}{\tanh(\vartheta t)} \big( |z|^2 + |w|^2 \big)
			\right)
			\exp\!\left(
			\frac{4\pi}{\sinh(\vartheta t)} \, S_{\mu t}^{-1} w \cdot z
			\right),
			\quad z, w \in \mathbb{R}^{2d}.
		\end{equation}
	\end{corollary}
	\begin{proof}
		Since $\cF_{\mu t}\in \Mp(d,\bR)$ and has projection $S_{\mu t}$, we know that its Wigner kernel is the distribution $k_{\cF_{\mu t}}(z,w)=\delta_{w=S_{\mu t}z}$. On the other hand, the Wigner kernel $k_W'$ of $\mathfrak{R}_{\vartheta t}$ is given by \eqref{defkprime2}, with $\Theta=(\vartheta t,\ldots,\vartheta t)\in\rd_{>0}$:
				\begin{equation}\label{defkprime3}
					k_W'(z,w)
					= [\sinh(\vartheta t)]^{-d} \,
					\exp\!\left(
					-\frac{2\pi}{\tanh(\vartheta t)} \big( |z|^2 + |w|^2 \big)
					\right)
					\exp\!\left(
					\frac{4\pi}{\sinh(\vartheta t)} \, z\cdot w
					\right).
				\end{equation}
		Therefore, by \eqref{compositionLaw}, we see that the Wigner kernel $k_W$ of $\mathfrak{R}_{(\vartheta+i\mu)t}$ is
	\begin{align}
		k_W(z,w)
		&= \int_{\mathbb{R}^{2d}} k_W'(z, \zeta) \, \delta_{w = S_{\mu t} \zeta} \, d\zeta
		= k_W'\!\left(z, S_{\mu t}^{-1} w\right) \\
		&= [\sinh(\vartheta t)]^{-d} \,
		\exp\!\left(
		-\frac{2\pi}{\tanh(\vartheta t)} \big( |z|^2 + |S_{\mu t}^{-1} w|^2 \big)
		\right)
		\exp\!\left(
		\frac{4\pi}{\sinh(\vartheta t)} \, S_{\mu t}^{-1} w \cdot z
		\right).
	\end{align}
		Since $S_{\mu t}$ is a rotation, the assertion follows.
	\end{proof}

\section*{Acknowledgements}
The  three authors have been supported by the Gruppo Nazionale per l’Analisi Matematica, la Probabilità e le loro Applicazioni (GNAMPA) of the Istituto Nazionale di Alta Matematica (INdAM). Gianluca Giacchi has been supported by the SNSF starting grant ``Multiresolution methods for unstructured data” (TMSGI2 211684).

\end{document}